\documentclass[11pt]{amsart}
\usepackage{amsmath}
\usepackage{amssymb}
\usepackage{esint}
\usepackage{tikz}
\usepackage{comment}

\theoremstyle{plain}
\newtheorem{theorem}{Theorem}[section]
\newtheorem{lemma}[theorem]{Lemma}

\newtheorem{cor}[theorem]{Corollary}

\theoremstyle{definition}

\numberwithin{equation}{section}

\def\be{\begin{equation}}
\def\ee{\end{equation}}

\begin{document}

\title[The Loewner-Nirenberg Problem in Cones]
{The Loewner-Nirenberg Problem in Cones}
\author[Han]{Qing Han}
\address{Department of Mathematics\\
University of Notre Dame\\
Notre Dame, IN 46556, USA} \email{qhan@nd.edu}
\author[Jiang]{Xumin Jiang}
\address{Department of Mathematics\\
Fordham University\\
Bronx, NY 10458, USA}  \email{xjiang77@fordham.edu}
\author[Shen]{Weiming Shen}
\address{School of Mathematical Sciences\\
Capital Normal University\\
Beijing, 100048, China}
\email{wmshen@pku.edu.cn}

\begin{abstract}
We study asymptotic behaviors of solutions to the 
Loewner-Nirenberg problem in finite cones and establish optimal 
asymptotic expansions in terms of the corresponding solutions in infinite cones. 
The spherical domains over which cones are formed are allowed to have singularities. 
An elliptic operator on such spherical domains with coefficients singular on boundary 
play an important role. Due to the singularity of the spherical domains, extra cares are needed 
for the study of the global regularity of the eigenfunctions and solutions of the 
associated singular Dirichlet problem. 
\end{abstract}


\maketitle


\section{Introduction}\label{sec-Intro}

In a pioneering work,
Loewner and Nirenberg \cite{Loewner&Nirenberg1974} studied the following 
problem which bears their names: 
\begin{align}\label{eq-LN-eq}
\Delta u&=\frac{1}{4}n(n-2)u^{\frac{n+2}{n-2}}\quad \text{in } \Omega,\\
\label{eq-LN-boundary}u&=\infty\quad\text{on }\partial\Omega,
\end{align}
where $\Omega$ is a bounded domain in $\mathbb{R}^n$, for $n\geq 3$. 
Under the condition that $\Omega$ is a $C^2$-domain, they proved the existence of 
a unique positive solution $u$ of \eqref{eq-LN-eq}-\eqref{eq-LN-boundary}
and established 
an estimate 
of $d_\Omega^{\frac{n-2}{2}}u$, where $d_\Omega$ is the distance function 
in $\Omega$ to $\partial\Omega$. Specifically, 
they proved,  for $d_\Omega$ sufficiently small,
\begin{equation}\label{eq-EstimateDegree1}|d_\Omega^{\frac{n-2}{2}}u-1|\le Cd_\Omega, \end{equation}
where $C$ is a positive constant depending only on certain geometric quantities of $\partial\Omega$.
The Loewner-Nirenberg problem
has a geometric interpretation. 
Its solution implies the existence of a complete conformal metric on $\Omega$ with the constant scalar curvature $-n(n-1)$.

The Loewner-Nirenberg problem has many generalizations. 
Aviles and McOwen \cite{AM1988DUKE} proved the existence of complete conformal metrics 
with a negative constant scalar curvature on compact manifolds with boundary. 
Guan \cite{Guan2008} and Gursky, Streets, and Warren \cite{M.Gursky1}
proved the existence of complete conformal metrics of negative Ricci
curvature on compact manifolds with boundary. 
Del Mar Gonzalez, Li, and Nguyen \cite{Gonzalez2018} studied the 
existence and uniqueness to a fully nonlinear version of the Loewner-Nirenberg problem.

The estimate \eqref{eq-EstimateDegree1} stimulated more studies of asymptotic behaviors of solutions $u$
of \eqref{eq-LN-eq}-\eqref{eq-LN-boundary}.  
Kichenassamy
\cite{Kichenassamy2005JFA} expanded further if $\Omega$ has a $C^{2,\alpha}$-boundary.
If $\Omega$ has a smooth boundary, 
Andersson, Chru\'sciel and Friedrich \cite{ACF1982CMP} and Mazzeo \cite{Mazzeo1991} established 
a polyhomogeneous expansion, 
an estimate up to an arbitrarily finite order.
All these results require $\partial\Omega$ to have
some degree of regularity. 
If $\Omega$ is a Lipschitz domain, Han and Shen \cite{hanshen2} studied 
asymptotic behaviors of solutions 
near singular points on $\partial\Omega$, and proved an estimate similar as 
\eqref{eq-EstimateDegree1},  under appropriate conditions 
of the domain near singular points. 

In this paper, we study a more basic question and investigate asymptotic behaviors of solutions of 
\eqref{eq-LN-eq}-\eqref{eq-LN-boundary} if $\Omega$ is a finite cone. 
Let  $V$ be an infinite cone over some spherical domain $\Sigma\subsetneq \mathbb{S}^{n-1}$. 
We always assume that the boundary $\partial\Sigma$ is $(n-2)$-dimensional. 
We consider \eqref{eq-LN-eq}-\eqref{eq-LN-boundary} in truncated cones. 
Let $u$ be a positive solution of 
\begin{align}\label{eq-LN-cone-truncated-eq}
\Delta u&=\frac{1}{4}n(n-2)u^{\frac{n+2}{n-2}}\quad \text{in } V\cap B_1,\\
\label{eq-LN-cone-truncated-boundary}u&=\infty\quad\text{on }\partial V\cap B_1.
\end{align}
We will study asymptotic behaviors of $u$ near the vertex of the cone $V$. 
To this end, we  introduce the corresponding solution in the infinite cone. 
Consider 
\begin{align}\label{eq-LN-cone-infinite-eq}
\Delta u_V&=\frac{1}{4}n(n-2)u_V^{\frac{n+2}{n-2}}\quad \text{in } V,\\
\label{eq-LN-cone-infinite-boundary}u_V&=\infty\quad\text{on }\partial V.
\end{align}
According to \cite{hanshen2}, under appropriate assumptions on $\Sigma$, 
there exists a unique positive solution $u_V$ of
\eqref{eq-LN-cone-infinite-eq}-\eqref{eq-LN-cone-infinite-boundary}.
Moreover, in the polar coordinates $x=r\theta$ with $r=|x|$ and $\theta=x/|x|$, 
$u_V$ has the form 
\begin{equation}\label{eq-relation-uV-u}u_V(x)=|x|^{-\frac{n-2}{2}}\xi_\Sigma(\theta),\end{equation} 
where  $\xi_\Sigma$ is a smooth function on $\Sigma$, with 
$\xi_\Sigma=\infty$ on $\partial \Sigma$.

Our primary goal in this paper is to investigate 
how a solution $u$ of \eqref{eq-LN-cone-truncated-eq}-\eqref{eq-LN-cone-truncated-boundary} 
in $V\cap B_1$  is approximated by the corresponding solution $u_V$ of 
\eqref{eq-LN-cone-infinite-eq}-\eqref{eq-LN-cone-infinite-boundary} restricted to $V\cap B_1$. 
We point out that both $u$ and $u_V$ are infinite on $\partial V\cap B_1$. It is not clear 
whether the quotient $u/u_V$ should remain bounded. In the first result, we prove that 
$u/u_V$ is not only bounded, but has an expansion near $\partial V\cap B_{1/2}$ with 
the constant leading term 1.


\begin{theorem}\label{thrm-main1} For $n\geq 3$, let $V$ be an infinite Euclidean cone over 
some Lipschitz domain $\Sigma\subsetneq\mathbb{S}^{n-1}$, and 
$u_V\in C^\infty(V)$ be the positive solution of 
\eqref{eq-LN-cone-infinite-eq}-\eqref{eq-LN-cone-infinite-boundary}.
Then, there exist a positive constant $\tau$ and  an increasing sequence of positive constants $\{\mu_i\}$, 
with $\mu_i\to\infty$, such that, for any positive solution $u\in {C}^\infty(V\cap B_1)$ of 
\eqref{eq-LN-cone-truncated-eq}-\eqref{eq-LN-cone-truncated-boundary} and 
any integer $m\ge 0$, 
\begin{equation}\label{expansion-negative}
   \Big|(u_V^{-1}u)(x)-1-\sum_{i=1}^{m}\sum_{j=0}^{i-1}c_{ij}(\theta)|x|^{\mu_i}(-\ln{|x|})^j\Big |   
     \leq Cd_\Sigma^{\tau}|x|^{\mu_{m+1}}(-\ln{|x|})^m,
\end{equation}
where $C$ is a positive constant, $d_\Sigma$ is the distance function on $\Sigma$ to $\partial \Sigma$, 
and each $c_{ij}$ is a bounded smooth function on $\Sigma$, with
$c_{ij}=O(d_\Sigma^{\tau})$. 
\end{theorem}

We point out that the sequence $\{\mu_i\}$ in Theorem \ref{thrm-main1} is determined only by the cone $V$ or 
the spherical domain $\Sigma$, independent of the specific solution $u$ of 
\eqref{eq-LN-cone-truncated-eq}-\eqref{eq-LN-cone-truncated-boundary}. 
The growth rate $\tau$ near $\partial\Sigma$ is also determined only by $\Sigma$ 
and will be proved to satisfy $\tau>\frac{n-2}2$. We can improve $\tau$ under appropriate enhanced 
assumptions of $\Sigma$. 

\begin{theorem}\label{thrm-main2} For $n\geq 3$, let $V$ be an infinite Euclidean cone over 
some domain $\Sigma\subsetneq\mathbb{S}^{n-1}$. 
Assume that  $\Sigma$
is either a  $C^{1,\alpha}$-domain for some $\alpha\in (0,1)$ or 
a Lipschitz domain  
which is the union of an increasing sequence of $C^{3}$-domains 
$\{\Sigma_i\}\subset\mathbb{S}^{n-1}$ such that 
each $\Sigma_i$ has a nonnegative mean curvature with respect to the inner unit normal. 
Then, $\tau$ in Theorem \ref{thrm-main1} is given by $\tau=n$. 
\end{theorem}

We point out that $\tau=n$ provides an optimal estimate near $\partial\Sigma$. 

For $m=0$ and $m=1$, 
\eqref{expansion-negative} reduces to 
\begin{equation}\label{expansion-negative-0}
\big|(u_V^{-1}u)(x)-1\big |   \leq Cd_\Sigma^{\tau}|x|^{\mu_1},
\end{equation}
and, for some $\mu> \mu_1$,  
\begin{equation}\label{expansion-negative-1}
\big|(u_V^{-1}u)(x)-1-c_1(\theta)|x|^{\mu_1}\big |   
\leq Cd_\Sigma^{\tau}|x|^{\mu},
\end{equation}
where $\mu_1$ is related to the first eigenvalue of a singular elliptic operator on $\Sigma$ 
associated with the function $\xi_\Sigma$ in \eqref{eq-relation-uV-u}, and can be computed explicitly. 
The estimate \eqref{expansion-negative-0} with $\tau=n$ is optimal, in terms of both powers  $\mu_1$ and $n$. 
In the case $V=\mathbb R_+^n$ (i.e., $\Sigma=\mathbb S_+^{n-1}$), we can prove 
that $\mu_1=n$. Then, \eqref{expansion-negative-0} with $\tau=n$ reduces to a more familiar form 
\begin{equation*}
\big|x_n^{\frac{n-2}2}u(x)-1\big |   \leq Cx_n^{n}.
\end{equation*}

Jiang \cite{jiang} proved Theorem \ref{thrm-main1}  with $\tau=n$ for the case that 
$\Sigma$ is a star-shaped smooth spherical domain, and also discussed 
asymptotic expansions near $\partial \Sigma$. 
If $\Sigma$ is at least $C^2$,  the distance function $d_\Sigma$ is at least $C^2$ near $\partial\Sigma$
and hence can be used to construct 
various barrier functions. This is the strategy adopted in \cite{jiang}. 

Under the weakened assumption  that $\Sigma$ is  Lipschitz,
the distance function $d_\Sigma$ is not $C^2$ near $\partial\Sigma$ 
anymore and cannot be used to construct barrier functions. 
We will construct barrier functions with the help of $\xi_\Sigma$, 
the smooth function  introduced in \eqref{eq-relation-uV-u}. 

To prove Theorems \ref{thrm-main1} and  \ref{thrm-main2}, consider 
$$v=|x|^{\frac{n-2}{2}}(u-u_V)=|x|^{\frac{n-2}{2}}u-\xi_\Sigma(\theta),$$ 
where $\xi_\Sigma$ is the function  on $\Sigma$ as in \eqref{eq-relation-uV-u}. 
We will linearize the equation of $|x|^{\frac{n-2}{2}}u$ at $|x|^{\frac{n-2}{2}}u_V$ 
and study its kernels. 
The corresponding linearized operator $\mathcal L$ 
is given by 
\begin{equation}\label{linearizedoperator-negative}
\mathcal Lw=r^2\partial_{rr}w+r\partial_rw+\Delta_{\theta}w-\frac{1}{4}n(n+2)\xi_\Sigma^{\frac{4}{n-2}}w
-\frac14(n-2)^2w.
\end{equation}
We need to discuss the eigenvalue problem for the linear operator 
\begin{align}\label{eq-linear-spherical}
    L_\Sigma w= \Delta_\theta w-\frac{1}{4}n(n+2)\xi_\Sigma^{\frac{4}{n-2}}w\quad\text{in }\Sigma.
\end{align}
Since $\xi_\Sigma$ blows up on boundary $\partial\Sigma$, $L_\Sigma$ has a singularity on $\partial\Sigma$. 
The singularity of some coefficient of $L_\Sigma$, coupled with the singularity of boundary of Lipschitz domains, 
makes the study of the asymptotic expansions a tricky issue.

If $\Sigma$ is merely Lipschitz, 
solutions related to the operator $\mathcal L$ in \eqref{linearizedoperator-negative} 
which are continuous up to boundary 
are proved to be H\"older continuous, but with a small H\"older index. 
Such smallness causes two serious issues. 
It does not allow us 
to perform some important computations such as integration by parts,   
and it causes some 
singular terms to possess less desirable integrability, especially in low dimensions. 
To overcome the first difficulty, we prove that derivatives of solutions with respect to $r$ are also H\"older continuous,
and then demonstrate that solutions restricted to each slice are in appropriate Sobolev spaces. 
For the second issue, we derive a 
universal lower bound estimate of 
the H\"older index in the case $n=3$, and hence improve the integrability to the desired level. 
Refer to Corollaries \ref{cor-regularity-v-H10} and \ref{cor-F(v)-integrability} for details, 
respectively.

We point out that Theorem \ref{thrm-main1} is not a prerequisite of Theorem  \ref{thrm-main2}, 
which can be proved directly and more easily. 
Moreover, both Theorem \ref{thrm-main1} and Theorem  \ref{thrm-main2} hold for $V=\mathbb R^{n-k}\times V_k$,  
where $V_k$ is a cone over an appropriate domain $\Sigma_{k-1}\subsetneq \mathbb S^{k-1}$ in $\mathbb R^k$. 

\smallskip 

We now compare results in this paper with corresponding results for the singular Yamabe equation
corresponding to positive scalar curvatures. We consider 
positive solutions of the Yamabe equation of the form
\begin{equation}\label{eq-Yamabe} 
-\Delta u=\frac14n(n-2)u^{\frac{n+2}{n-2}},\end{equation}
which are defined in the punctured ball $B_1\setminus\{0\}$, with a nonremovable singularity at the origin. 

Let $u$ be a positive solution of \eqref{eq-Yamabe} in $B_1\setminus \{0\}$, 
with a nonremovable singularity at the origin. 
In a pioneering paper \cite{CaffarelliGS1989}, Caffarelli,  Gidas,  and Spruck 
proved that $u$ is asymptotic to a radial singular solution of \eqref{eq-Yamabe} in $\mathbb R^n\setminus\{0\}$;  
namely, 
\begin{equation}\label{eq-estimate-u-0}
|x|^{\frac{n-2}2}u(x)-\psi(-\ln|x|)\to 0\quad\text{as }x\to 0,\end{equation}
where $|x|^{\frac{2-n}2}\psi(-\ln|x|)$ is a positive radial solution of \eqref{eq-Yamabe} 
in $\mathbb R^n\setminus\{0\}$, with a nonremovable singularity at the origin. In fact, 
$\psi$ is a positive periodic function in $\mathbb R$. 
Subsequently in \cite{KorevaarMPS1999}, Korevaar, Mazzeo, Pacard, and Schoen 
studied refined asymptotics and expanded solutions to the next order in the following form: 
for some constant $\alpha\in (1,2]$, 
\begin{equation}\label{eq-estimate-u-1}
\big||x|^{\frac{n-2}2}u(x)-\psi(-\ln|x|)-\phi(-\ln|x|)P_1(x)\big|\le C|x|^{\alpha}\quad\text{in }B_{1/2},\end{equation}
where $\psi$ is the function as in \eqref{eq-estimate-u-0}, 
$P_1$ is a linear function, and $\phi$ is a function given by 
$$\phi=-\psi'+\frac{n-2}2\psi.$$
In \cite{Han-Li-Li}, Han, Li, and Li established an expansion of $|x|^{\frac{n-2}2}u(x)$ up to arbitrary orders. 
Specifically, there exists a 
positive sequence $\{\mu_i\}$, strictly increasing, divergent to $\infty$ and with $\mu_1=1$, 
such that, for any positive integer $m$ 
and any $x\in B_{1/2}\setminus\{0\}$, 
\begin{align}\label{eq-estimate-u-k}
\Big||x|^{\frac{n-2}2}u(x)-\psi(-\ln|x|)-\sum_{i=1}^m\sum_{j=0}^{i-1} 
c_{ij}(x)|x|^{\mu_i}(-\ln|x|)^j\Big|
\le C|x|^{\mu_{m+1}}(-\ln|x|)^m,
\end{align}
where  $\psi$ is the function as in \eqref{eq-estimate-u-0}, and
$c_{ij}$ is a bounded smooth function in $B_{1/2}\setminus\{0\}$, 
for each $i=1, \cdots, m$ and $j=0,\cdots, i-1$. We point out that the sequence 
$\{\mu_i\}$ here is determined by the leading term $\psi$, and hence by the solution $u$. 

Compare \eqref{expansion-negative-0}, \eqref{expansion-negative-1}, and \eqref{expansion-negative}
with \eqref{eq-estimate-u-0}, \eqref{eq-estimate-u-1}, and \eqref{eq-estimate-u-k}, respectively. 

\smallskip

The paper is organized as follows. 
In Section \ref{sec-Existence-Cones}, we discuss some basic properties of solutions in cones. 
In Section \ref{sec-gradient}, we derive some necessary gradient estimates. 
In Section \ref{sec-spherical-operators}, we study the eigenvalue problem 
of the elliptic operator $-L_\Sigma$ introduced in \eqref{eq-linear-spherical}. 
In Section \ref{sec-boundary-regularity}, 
we discuss the regularity of solutions of the Yamabe equation 
near boundary of cylinders. 
In Section \ref{section-optimal-t}, we derive an optimal estimate along the $t$-direction. 
In Section \ref{sec-iterations}, we discuss asymptotic expansions for large $t$
and prove Theorems \ref{thrm-main1} and \ref{thrm-main2}. 

W. Shen acknowledges the partial support by NSFC Grant 11901405.

\section{Solutions in Cones}\label{sec-Existence-Cones}

In this section,
we discuss the existence of solutions of the Loewner-Nirenberg problem in infinite cones, 
and compare solutions in infinite cones and in finite cones. 

First, we quote a well-known result.

\begin{theorem}\label{thrm-Existence}
Let $ \Omega $ be a bounded Lipschitz domain in $ \mathbb{R}^{n}$.
Then, there exists a unique positive solution  $ u \in C^{ \infty}(\Omega)$
of \eqref{eq-LN-eq}-\eqref{eq-LN-boundary}.
\end{theorem}

Loewner and Nirenberg \cite{Loewner&Nirenberg1974} proved the existence and uniqueness
for $C^2$-domains. 
Refer to \cite{LazerMcKenna1993} for the general case.

Now, we state a  basic result which will be needed later.

\begin{lemma}\label{lemma-supersution}
Let $ \Omega $ be a domain in $ \mathbb{R}^{n}$,
and $u$ and $v$ be two nonnegative solutions of  \eqref{eq-LN-eq}.
Then, $u+v$ is a
nonnegative supersolution of  \eqref{eq-LN-eq}.
\end{lemma}

We omit the proof as it is based on a straightforward calculation.

\smallskip

Next, we discuss \eqref{eq-LN-eq}-\eqref{eq-LN-boundary} 
in infinite cones. Throughout this paper, cones are always solid.
Let $(r,\theta)$ be the polar coordinates in $\mathbb R^n$. Then,
\begin{equation}\label{eq-Delta-identity}\Delta=
\partial_{rr}+\frac{n-1}{r}\partial_r+\frac{1}{r^{2}}\Delta_{\theta},\end{equation}
where $\Delta_{\theta}$ is the  Laplace-Beltrami operator on the unit
sphere $  \mathbb S ^{n-1}$. 

Suppose $u$ is a positive function and set 
$$\widehat u(x)=|x|^{\frac{n-2}2}u(x).$$
Then, $u$ is a solution of \eqref{eq-LN-eq} if and only if 
\begin{equation}\label{eq-LN-polar} 
r\partial_r(r\partial_r\widehat u)
+\Delta_\theta \widehat u-\frac14(n-2)^2\widehat u
=\frac14n(n-2)\widehat u^{\frac{n+2}{n-2}}.\end{equation}

\begin{theorem}\label{thrm-LN-existence-cone} 
Let ${V}$ be an infinite cone in $\mathbb R ^{n}$ over a 
Lipschitz domain $\Sigma\subsetneq\mathbb S^{n-1}$.
Then, there exists a unique positive solution $u_V\in C^\infty(V)$ of
\eqref{eq-LN-cone-infinite-eq}-\eqref{eq-LN-cone-infinite-boundary}. 
\end{theorem}

\begin{proof} The proof consists of three steps. 

{\it Step 1.} We  prove the existence. 
Set
\begin{equation}\label{eq-solution-u-product}u_V(x)=|x|^{-\frac{n-2}{2}} \xi(\theta).\end{equation}
By \eqref{eq-LN-polar}, $u_V$ satisfies 
\eqref{eq-LN-cone-infinite-eq}-\eqref{eq-LN-cone-infinite-boundary}
if 
\begin{align}\label{eq-LN-spherical-domain-eq}
    \Delta_{\theta}\xi-\frac{1}{4}(n-2)^2\xi&=\frac{1}{4}n(n-2)\xi^{\frac{n+2}{n-2}}\quad\text{in } \Sigma,\\
\label{eq-LN-spherical-domain-boundary}    \xi&=\infty\quad\text{on }\partial \Sigma.
\end{align}
By a similar method as the proof of Theorem \ref{thrm-Existence}, there exists a unique positive solution $\xi$ 
of \eqref{eq-LN-spherical-domain-eq}-\eqref{eq-LN-spherical-domain-boundary}. 
Hence, there exists a positive solution $u$ of 
\eqref{eq-LN-cone-infinite-eq}-\eqref{eq-LN-cone-infinite-boundary}. In fact, for each $i\ge 1$, we consider 
\begin{align}\label{eq-LN-spherical-domain-eq-i}
    \Delta_{\theta}\xi_i-\frac{1}{4}(n-2)^2\xi_i&=\frac{1}{4}n(n-2)\xi_i^{\frac{n+2}{n-2}}\quad\text{in } \Sigma,\\
\label{eq-LN-spherical-domain-boundary-i}    \xi&=i\quad\text{on }\partial \Sigma.
\end{align}
It is standard to prove that there exists a solution $\xi_i\in C^\infty(\Sigma)\cap C(\bar\Sigma)$ of 
\eqref{eq-LN-spherical-domain-eq-i}-\eqref{eq-LN-spherical-domain-boundary-i} and that 
$\{\xi_i\}$ is a monotone increasing sequence, and hence converges to some $\xi$, a solution of 
\eqref{eq-LN-spherical-domain-eq}-\eqref{eq-LN-spherical-domain-boundary}.
We point out that the uniqueness of $\xi$ 
shows that the solution of \eqref{eq-LN-cone-infinite-eq}-\eqref{eq-LN-cone-infinite-boundary} 
in the form \eqref{eq-solution-u-product} is unique.

{\it Step 2.} We introduce some notations.
Let $x_{0}\in \mathbb R^n$ be a point and $r, R>0$ be  constants.
Set, for any $x\in B_R(x_0)$,
\begin{equation}\label{eq-SolutionInside}
u_{R, x_0}(x)=\Big(\frac{2R}{R^{2}-|x-x_0|^{2}}\Big)^{\frac{n-2}{2}}.\end{equation}
Then, $u_{R,x_0}$ is the solution of
\eqref{eq-LN-eq}-\eqref{eq-LN-boundary}  in $B_R(x_0)$.
Set, for any  $x\in\mathbb R^n\setminus B_r(x_0)$,
\begin{equation}\label{eq-SolutionOutside}
v_{r,x_0}(x)=\Big(\frac{2r}{|x-x_0|^{2}-r^{2}}\Big)^{\frac{n-2}{2}}.\end{equation}
Then, $v_{r,x_0}$ is a solution of
\eqref{eq-LN-eq}-\eqref{eq-LN-boundary}  in $\mathbb R^n\setminus B_r(x_0)$.

For any fixed $x$, we have $x\in B_R(x_0)$ for all large $R$, and $u_{R, x_0}(x) \to0$ as $R\to\infty$. 
Similarly, for any fixed $x\neq x_0$, we have $x\in \mathbb R^n\setminus B_r(x_0)$ for all small $r$, 
and $v_{r, x_0}(x) \to0$ as $r\to0$. 

{\it Step 3.} We now prove the uniqueness. Let $u_V$ be the solution of 
\eqref{eq-LN-cone-infinite-eq}-\eqref{eq-LN-cone-infinite-boundary} established in Step 1, given by 
\eqref{eq-solution-u-product}, and $u$ be an any other solution of 
\eqref{eq-LN-cone-infinite-eq}-\eqref{eq-LN-cone-infinite-boundary}. 

Take any $0<r<R$. By Lemma \ref{lemma-supersution} and the maximum principle, we have 
$$|x|^{-\frac{n-2}2}\xi_i\le u+u_{R,0}+v_{r, 0}\quad\text{in }V\cap (B_R\setminus \bar B_r),$$ 
where $\xi_i$ is the solution of \eqref{eq-LN-spherical-domain-eq-i}-\eqref{eq-LN-spherical-domain-boundary-i}, 
and $u_{R,0}$ and $v_{r,0}$ are given by \eqref{eq-SolutionInside} and \eqref{eq-SolutionOutside}, 
respectively, for $x_0=0$. Letting $i\to \infty$, we get 
\begin{equation}\label{eq-estimate-one-side}
u_V\le  u+u_{R,0}+v_{r, 0}\quad\text{in }V\cap (B_R\setminus \bar B_r).\end{equation} 
Next, we consider a sequence of Lipschitz domains $\Sigma_k\subset\subset\Sigma$ such that $\Sigma_k\to \Sigma$, 
and denote by $V_k$ the cone over $\Sigma_k$. Then, $V_k\subsetneq V$. 
For each $k$, let $\xi^{(k)}$ be the solution of 
\eqref{eq-LN-spherical-domain-eq}-\eqref{eq-LN-spherical-domain-boundary} in $\Sigma_k$. 
Then, $\xi^{(k)}\to \xi$ uniformly locally in $\Sigma$. 
Similarly, by Lemma \ref{lemma-supersution} and the maximum principle, we have 
$$ u\le |x|^{-\frac{n-2}2}\xi^{(k)}+u_{R,0}+v_{r, 0}\quad\text{in }V_k\cap (B_R\setminus \bar B_r).$$ 
Letting $k\to \infty$, we get 
\begin{equation}\label{eq-estimate-another-side} 
u\le u_V+u_{R,0}+v_{r, 0}\quad\text{in }V\cap (B_R\setminus \bar B_r).\end{equation}
By combining \eqref{eq-estimate-one-side} and \eqref{eq-estimate-another-side}, we obtain 
\begin{equation}\label{eq-estimate-both-sides} 
|u- u_V|\le u_{R,0}+v_{r, 0}\quad\text{in }V\cap (B_R\setminus \bar B_r).\end{equation}
For any fixed $x\in V$, take $r<|x|<R$, and then let $R\to\infty$ and $r\to 0$. 
We conclude that $u=u_V$ in $V$. 
\end{proof}

The proof is modified from \cite{hanshen2}, with a new proof of the uniqueness. 
In the present version, we removed the requirement in \cite{hanshen2} that $\Sigma$ is star-shaped. 
The function $\xi$ introduced in \eqref{eq-solution-u-product} plays an important role in this paper. 
We now present some of its properties. In the next result, denote by $d$  
the distance function on $\Sigma$ to $\partial \Sigma$.

\begin{lemma}\label{lemma-solutions-xi-spherical} 
Let $\Sigma\subsetneq\mathbb S^{n-1}$ be a 
Lipschitz domain.
Then, there exists a unique positive solution $\xi\in C^\infty(\Sigma)$ of
\eqref{eq-LN-spherical-domain-eq}-\eqref{eq-LN-spherical-domain-boundary}. 
Moreover, 
\begin{equation}\label{eq-estimate-xi-upper-lower} 
c_1\le d^{\frac{n-2}2}\xi\le c_2\quad\text{in }\Sigma,\end{equation}
and, for any $k\ge 0$,  
\begin{equation}\label{eq-estimate-xi-derivatives} 
d^{\frac{n-2}2+k}|\nabla_\theta^k \xi|\le C\quad\text{in }\Sigma,\end{equation}
where $c_1$, $c_2$, and $C$ are positive constants depending only on $n$, $k$ and $\Sigma$. 
If, in addition, $\Sigma$ is a $C^{1,\alpha}$-domain for some $\alpha\in (0,1]$, then 
\begin{equation}\label{eq-C1alpha-spherical}
|d^{\frac{n-2}2}\xi-1|\le Cd^\alpha\quad\text{in }\Sigma.\end{equation}
\end{lemma}

The existence and the uniqueness of the solution $\xi$ are proved in the proof of 
Theorem \ref{thrm-LN-existence-cone}. 
Estimates similar as \eqref{eq-estimate-xi-upper-lower}-\eqref{eq-C1alpha-spherical} 
for solutions $u$ of \eqref{eq-LN-eq}-\eqref{eq-LN-boundary} are well-known, and proofs are standard. 
(Refer to \cite{hanshen2}.) 
These proofs can be modified easily 
to yield \eqref{eq-estimate-xi-upper-lower}-\eqref{eq-C1alpha-spherical}
for solutions $\xi$ of \eqref{eq-LN-spherical-domain-eq}-\eqref{eq-LN-spherical-domain-boundary}.


\smallskip 

For later purposes, we present an equivalent form of 
\eqref{eq-LN-spherical-domain-eq}-\eqref{eq-LN-spherical-domain-boundary}. 
Consider
\begin{align}\label{eq-LN-domain-eq-rho-new}
    \rho\Delta_\theta \rho
    +S\rho^2&=\frac{n}{2}(|\nabla_\theta \rho|^2-1)\quad\text{in } \Sigma,\\
\label{eq-LN-domain-rho-boundary-new}    \rho&=0\quad\text{on }\partial \Sigma,
\end{align}
where $S$ is the constant given by 
\begin{align}\label{eq-constant-S}
S=\frac12(n-2).\end{align}
In the equation \eqref{eq-LN-domain-eq-rho-new}, we purposely introduce the constant $S$, 
which is related to the scalar curvature of $\mathbb S^{n-1}$. 
The equation in the Euclidean space corresponds to $S=0$. 

\begin{lemma}\label{lemma-solutions-rho-spherical} 
Let $\Sigma\subsetneq\mathbb S^{n-1}$ be a 
Lipschitz domain.
Then, there exists a unique $\rho\in C^\infty(\Sigma)\cap \mathrm{Lip}(\Sigma)$, 
positive in $\Sigma$ and satisfying \eqref{eq-LN-domain-eq-rho-new}-\eqref{eq-LN-domain-rho-boundary-new}. 
Moreover, 
\begin{equation}\label{eq-estimate-rho-upper-lower0} 
c_1\le \frac{\rho}{d}\le c_2\quad\text{in }\Sigma,\end{equation}
where $c_1$ and $c_2$ are positive constants depending only on $n$ and $\Sigma$. 
\end{lemma}

\begin{proof} Let $\xi\in C^\infty(\Sigma)$ be 
the positive solution of 
\eqref{eq-LN-spherical-domain-eq}-\eqref{eq-LN-spherical-domain-boundary} as in 
Lemma \ref{lemma-solutions-xi-spherical}. Set
\begin{equation}\label{eq-relation-rho-xi}\rho=\xi^{-\frac2{n-2}}.\end{equation} 
Then, $\rho\in C^\infty(\Sigma)\cap C(\bar \Sigma)$ is a positive function in $\Sigma$ and 
satisfies \eqref{eq-LN-domain-eq-rho-new}-\eqref{eq-LN-domain-rho-boundary-new}. 
By \eqref{eq-estimate-xi-upper-lower}, $\rho$ satisfies 
\eqref{eq-estimate-rho-upper-lower0} for some positive constants 
$c_1$ and $c_2$  depending only on $n$ and $\Sigma$. 
Next, by \eqref{eq-estimate-xi-derivatives} with $k=1$, we have 
$|\nabla_\theta \rho|\le C$. Hence, $\rho$ is Lipschitz in $\Sigma$. 
\end{proof}


We now return to the equation 
\eqref{eq-LN-eq} and prove a simple lemma comparing solutions in infinite cones and in truncated cones.

\begin{lemma}\label{lemma-estimate-bound} 
Let $V$ be an infinite cone in $\mathbb R^{n}$ over some
Lipschitz domain $\Sigma\subsetneq\mathbb{S}^{n-1}$.  
Suppose $u\in {C}^\infty(V\cap B_1)$ is a positive solution of 
\eqref{eq-LN-cone-truncated-eq}-\eqref{eq-LN-cone-truncated-boundary}
and $u_V\in {C}^\infty(V)$ is the positive solution 
of  \eqref{eq-LN-cone-infinite-eq}-\eqref{eq-LN-cone-infinite-boundary}. 
Then, 
\begin{equation}\label{eq-estimate-0-x}|u-u_V|\le C\quad\text{in }V\cap B_{1/2},\end{equation}
where $C$ is a positive constant depending only on $n$. 
\end{lemma}

\begin{proof} Proceeding similarly as in Step 3 in the proof of Theorem \ref{thrm-LN-existence-cone}, we have, 
for $r<1$,  
\begin{equation*} 
|u- u_V|\le u_{1,0}+v_{r, 0}\quad\text{in }V\cap (B_1\setminus \bar B_r).\end{equation*}
This is \eqref{eq-estimate-both-sides} with $R=1$. 
For any fixed $x\in V\cap B_1$, take $r<|x|$, and then let  $r\to 0$. 
We conclude 
\begin{equation*}
|u- u_V|<u_{1,0}\quad\text{in }V\cap B_1.\end{equation*}
This yields the desired result if we restrict to $V\cap B_{1/2}$. 
\end{proof}

Our goal in this paper is to improve the estimate \eqref{eq-estimate-0-x}. 
Let $u\in {C}^\infty(V\cap B_1)$ and $u_V\in {C}^\infty(V)$ be 
as in Lemma \ref{lemma-estimate-bound}.  
Set 
\begin{equation}\label{eq-def-v-polar}v=|x|^{\frac{n-2}2}(u-u_V)=|x|^{\frac{n-2}2}u-\xi,\end{equation}
where $\xi$ satisfies \eqref{eq-LN-spherical-domain-eq}-\eqref{eq-LN-spherical-domain-boundary}.  
Since both $|x|^{\frac{n-2}2}u$ and $|x|^{\frac{n-2}2}u_V$ satisfy the same equation  
\eqref{eq-LN-polar}, we can take a difference of these two equations and obtain 
\begin{align*} 
r\partial_r(r\partial_rv)
+\Delta_\theta v-\frac14(n-2)^2v
=\frac14n(n-2)\big[(v+\xi)^{\frac{n+2}{n-2}}-\xi^{\frac{n+2}{n-2}}\big].
\end{align*} 
For the right-hand side, a simple computation yields 
\begin{align*}\mathrm{RHS}&=\frac14n(n-2)\xi^{\frac{n+2}{n-2}}
\big[(1+\xi^{-1}v)^{\frac{n+2}{n-2}}-1\big]\\
&=\frac14n(n-2)\xi^{\frac{n+2}{n-2}}\big[(1+\xi^{-1}v)^{\frac{n+2}{n-2}}
-1-\frac{n+2}{n-2}\xi^{-1}v
+\frac{n+2}{n-2}\xi^{-1}v\big]\\
&=\xi^{-\frac{n-6}{n-2}}v^2h\big(\xi^{-1}v\big)+\frac14n(n+2)\xi^{\frac{4}{n-2}}v, 
\end{align*} 
where
$$h(s)=\frac14n(n-2)s^{-2}\Big[(1+s)^{\frac{n+2}{n-2}}-1-\frac{n+2}{n-2}s\Big].$$
Hence, 
\begin{equation*}
r\partial_r(r\partial_rv)+\Delta_\theta v-\frac14n(n+2)\xi^{\frac{4}{n-2}}v
-\frac14(n-2)^2v=
\xi^{-\frac{n-6}{n-2}}v^2h\big(\xi^{-1}v\big).\end{equation*}
Let $\rho$ be given by \eqref{eq-relation-rho-xi}. Then, 
\begin{equation}\label{eq-Yamabe-equivalent} 
r\partial_r(r\partial_rv)+\Delta_\theta v-\frac14n(n+2)\frac{v}{\rho^2}
-\frac14(n-2)^2v=
\rho^{\frac{n-6}2}v^2h\big(\rho^{\frac{n-2}2}v\big).\end{equation}
We point out that the left-hand side of \eqref{eq-Yamabe-equivalent} is linear in $v$ 
and the right-hand side is nonlinear. 
We also note that  $h$ is a smooth function on $(-1,1)$ and $h(0)\neq 0$. 

In the rest of the paper, we study the equation \eqref{eq-Yamabe-equivalent} in cylindrical coordinates. 
For any $x\in\mathbb R^n\setminus\{0\}$, 
set $(t,\theta)\in \mathbb R\times \mathbb S^{n-1}$ by
\begin{equation}\label{eq-change-coordinates}
t=-\ln|x|,\quad \theta=\frac{x}{|x|}.\end{equation}
In cylindrical coordinates, we write \eqref{eq-def-v-polar} as
\begin{equation}\label{eq-def-v-spherical}v(t, \theta)=|x|^{\frac{n-2}2}\big(u(x)-u_V(x)\big)
=|x|^{\frac{n-2}2}u(x)-\xi(\theta)=|x|^{\frac{n-2}2}u(x)-\rho^{-\frac{n-2}2}(\theta).\end{equation}
With the change of coordinates $r=e^{-t}$, we write \eqref{eq-Yamabe-equivalent} as 
\begin{equation}\label{eq-basic-equation-v}
\mathcal Lv=F(v),\end{equation} 
where 
\begin{equation}\label{eq-def-mathcal-L}
\mathcal Lv=\partial_{tt}v
+\Delta_\theta v-\frac14n(n+2)\frac{v}{\rho^2}-\frac14(n-2)^2v,\end{equation}
and 
\begin{equation}\label{eq-def-mathcal-F}
F(v)=\rho^{\frac{n-6}2}v^2
h\big(\rho^{\frac{n-2}2}v\big).\end{equation}
By Lemma \ref{lemma-estimate-bound} and \eqref{eq-def-v-spherical}, we have, 
for any $(t,\theta)\in (1, \infty)\times \Sigma$,  
\begin{equation}\label{eq-v1-0}|v(t,\theta)|\le Ce^{-\frac{n-2}{2} t},\end{equation}
where $C$ is a positive constant, 
depending on $n$ and $\Sigma$.
By \eqref{eq-v1-0}, $v$ is bounded and decays to zero exponentially as $t\to\infty$. 

\section{Gradient Estimates}\label{sec-gradient} 

Gradient estimates of solutions $\rho$ of 
\eqref{eq-LN-domain-eq-rho-new}-\eqref{eq-LN-domain-rho-boundary-new} play an important role in this paper. 
In this section, we present some of these estimates which will be needed later on. 

In general, a gradient bound of  $\rho$ depends  on the dimension $n$ and 
the size of the exterior cone at each boundary point, the opening angle and the height. 
In the case $n=3$, there is a universal upper bound. 
For $n=3$, \eqref{eq-LN-domain-eq-rho-new} and \eqref{eq-LN-domain-rho-boundary-new} reduce to 
\begin{align}\label{eq-LN-domain-eq-rho-new-n3}
    \rho\Delta_{\mathbb S^2} \rho
    +\frac12\rho^2&=\frac{3}{2}(|\nabla_{\mathbb S^2} \rho|^2-1)\quad\text{in } \Sigma,\\
\label{eq-LN-domain-rho-boundary-new-n3}    \rho&=0\quad\text{on }\partial \Sigma.
\end{align}

\begin{lemma}\label{lemma-n3-rho-upper-bound}
Let $\Sigma\subsetneq\mathbb{S}^{2}$ be a  
Lipschitz domain, and $\rho\in C^\infty(\Sigma)\cap \mathrm{Lip}(\bar\Sigma)$ 
be the positive solution of  
\eqref{eq-LN-domain-eq-rho-new-n3}-\eqref{eq-LN-domain-rho-boundary-new-n3} in $\Sigma$. 
Then, $\rho< 4.2$ in $\Sigma$. 
\end{lemma} 

\begin{proof} Let $\xi$ be given by \eqref{eq-relation-rho-xi}, i.e., 
\begin{equation}\label{eq-relation-rho-xi-n3}\rho=\xi^{-2}.\end{equation} 
Then, \eqref{eq-LN-spherical-domain-eq} and \eqref{eq-LN-spherical-domain-boundary} reduce to 
\begin{align}\label{eq-LN-spherical-domain-eq-n3}
    \Delta_{\mathbb S^2}\xi-\frac{1}{4}\xi&=\frac{3}{4}\xi^{5}\quad\text{in } \Sigma,\\
\label{eq-LN-spherical-domain-boundary-n3}    \xi&=\infty\quad\text{on }\partial \Sigma.
\end{align}

We denote by $(\theta,\varphi)$ the spherical coordinates on $\mathbb S^2$, 
with $\theta=0$ and $\theta=\pi$ corresponding to the north pole
and the south pole, respectively. 
We first construct a subsolution 
of \eqref{eq-LN-spherical-domain-eq-n3} in the 
compliment of the north pole. 
For any function $\eta=\eta(\theta)$, we have 
$$\Delta_{\mathbb S^2}\eta=\frac{1}{\sin\theta}\partial_\theta\big(\sin\theta\, \partial_\theta h\big).$$
Hence, for any constant $\tau$, we get 
$$\Delta_{\mathbb S^2}\big(\sin\frac\theta2\big)^\tau=\frac{1}{4}\tau^2\big(\sin\frac\theta2\big)^{\tau-2}
-\frac14\tau(\tau+2)\big(\sin\frac\theta2\big)^\tau.$$
With $c=(12)^{-\frac14}$, set 
$$\eta_1(\theta)=c\big(\sin\frac\theta2\big)^{-\frac12}, \quad \eta_2(\theta)=c\big(\sin\frac\theta2\big)^{\frac12}.$$
Take a positive constant $\alpha$ to be determined. 
Then, 
\begin{align}\label{eq-identity-eta12}
\Delta_{\mathbb S^2}(\eta_1-\alpha \eta_2)= \frac{c}{16}\Big[\big(\sin\frac\theta2\big)^{-\frac52}
-\alpha\big(\sin\frac\theta2\big)^{-\frac32}
+3\big(\sin\frac\theta2\big)^{-\frac12}
+5\alpha\big(\sin\frac\theta2\big)^{\frac12}\Big].\end{align}
We will find $\alpha$ such that 
\begin{align}\label{eq-subsolution-eta12}
\Delta_{\mathbb S^2}(\eta_1-\alpha \eta_2)> \frac34(\eta_1-\alpha \eta_2)^5
+\frac{1}{4}(\eta_1-\alpha \eta_2).\end{align}
To this end, we first note 
\begin{align*}
(\eta_1-\alpha \eta_2)^5&=c^5\big(\sin\frac\theta2\big)^{-\frac52}\Big[1-\alpha\sin\frac\theta2\Big]^5\\
&< c^5\big(\sin\frac\theta2\big)^{-\frac52}\Big[1-5\alpha\sin\frac\theta2+10\alpha^2\big(\sin\frac\theta2\big)^2\Big]\\
&=c^5\big(\sin\frac\theta2\big)^{-\frac52}-5c^5\alpha\big(\sin\frac\theta2\big)^{-\frac32}
+10c^5\alpha^2\big(\sin\frac\theta2\big)^{-\frac12},
\end{align*}
where we used the inequality $(1-x)^5<1-5x+10x^2$ for $x\in (0,1)$. 
By $c^4=(12)^{-1}$, we have 
\begin{align}\label{eq-subsolution-relation}\begin{split}
&\frac34(\eta_1-\alpha \eta_2)^5
+\frac{1}{4}(\eta_1-\alpha \eta_2)\\
&\qquad< \frac{c}{16}\Big[\big(\sin\frac\theta2\big)^{-\frac52}
-5\alpha\big(\sin\frac\theta2\big)^{-\frac32}
+(10\alpha^2+4)\big(\sin\frac\theta2\big)^{-\frac12}
-4\alpha\big(\sin\frac\theta2\big)^{\frac12}\Big].
\end{split}\end{align}
In order to have \eqref{eq-subsolution-eta12}, by comparing \eqref{eq-identity-eta12} and 
\eqref{eq-subsolution-relation}, we require 
$$4\alpha\big(\sin\frac\theta2\big)^{-\frac32}
-(10\alpha^2+1)\big(\sin\frac\theta2\big)^{-\frac12}
+9\alpha\big(\sin\frac\theta2\big)^{\frac12}\ge 0,$$
or 
$$4\alpha
-(10\alpha^2+1)\sin\frac\theta2
+9\alpha\big(\sin\frac\theta2\big)^2\ge 0.$$
To this end, we require 
$$10\alpha^2+1<12\alpha.$$
We can take $\alpha=1/11$. In summary,  set, for $\theta\neq 0$, 
\begin{equation}\label{eq-expression-subsolution}\eta(\theta)=\frac{1}{\sqrt[4]{12}}
\Big[\big(\sin\frac\theta2\big)^{-\frac12}-\frac1{11}\big(\sin\frac\theta2\big)^{\frac12}\Big]
=\frac{1}{\sqrt[4]{12}}
\big(\sin\frac\theta2\big)^{-\frac12}\Big[1-\frac1{11}\sin\frac\theta2\Big].\end{equation}
Then, 
$$\Delta_{\mathbb S^2}\eta> \frac34\eta^5+\frac14\eta.$$

Let $\rho$ and $\xi$ be the solution of 
\eqref{eq-LN-domain-eq-rho-new-n3}-\eqref{eq-LN-domain-rho-boundary-new-n3} and
\eqref{eq-LN-spherical-domain-eq-n3}-\eqref{eq-LN-spherical-domain-boundary-n3} 
on $\Sigma$, respectively. 
Denoting by $N$ the north pole on $\mathbb S^2$, we assume $N\notin \Sigma$. 
(Here, we allow $\Sigma=\mathbb S^2\setminus\{N\}$.) 
Consider a sequence of increasing domains $\{\Sigma_i\}\subset \mathbb S^2$ with 
$\cup \Sigma_i=\Sigma$, and let $\xi_i$ be the solution of 
\eqref{eq-LN-spherical-domain-eq-n3}-\eqref{eq-LN-spherical-domain-boundary-n3} on $\Sigma_i$. 
Then, $\xi_i\to \xi$ uniformly in any compact subset of $\Sigma$. 
By the maximum principle, we have $\xi_i\ge \eta$ on $\Sigma_i$, and hence $\xi\ge \eta$ on $\Sigma$. 
With \eqref{eq-relation-rho-xi-n3}, we have 
$$\rho\le \eta^{-2}\quad\text{on }\Sigma.$$
By \eqref{eq-expression-subsolution}, the maximum of $\eta^{-2}$ is attained at $\theta=\pi$ 
with a value $\sqrt{12}(11/10)^2<4.2$. 
\end{proof} 

\begin{lemma}\label{lemma-n3-gradient}
Let $\Sigma\subsetneq\mathbb{S}^{2}$ be a  
Lipschitz domain, and $\rho\in C^\infty(\Sigma)\cap \mathrm{Lip}(\bar\Sigma)$ 
be the positive solution of  
\eqref{eq-LN-domain-eq-rho-new-n3}-\eqref{eq-LN-domain-rho-boundary-new-n3} in $\Sigma$. 
Then, $|\nabla_\theta\rho|< 2.8$ in $\Sigma$. 
\end{lemma}

\begin{proof} For brevity, we write $\nabla$ and $\Delta$, instead of $\nabla_{\mathbb S^2}$ and 
$\Delta_{\mathbb S^2}$. We only consider the case that $\Sigma$ has a $C^2$-boundary, 
and obtain the general case by a simple approximation. 
Set $\bar\rho=4.2$, the universal upper bound established in Lemma \ref{lemma-n3-rho-upper-bound}. 

{\it Step 1.} We claim 
\begin{equation}\label{eq-gradient-claim1}
|\nabla\rho|^2+\frac{3+\bar\rho^2}{2\bar\rho^2}\rho^2\le \frac{3+\bar\rho^2}{2}\quad\text{on }\Sigma.\end{equation}
For some  constant $A\in [0,1]$ to be determined, set 
\begin{equation}\label{eq-definition-w}w=|\nabla \rho|^2+A\rho^2\quad\text{on }\Sigma.\end{equation} 
It is obvious that $w=1$ on $\partial\Sigma$. Assume $w$ attains its maximum at some $\theta_0\in\Sigma$. 
Then, $\nabla w=0$ and $\Delta w\le 0$ at $\theta_0$, i.e., 
\begin{equation}\label{eq-1-derivative-zero}\rho_i\rho_{ij}+A\rho\rho_j=0\quad\text{at }\theta_0,\end{equation}
and 
\begin{equation}\label{eq-2-derivative-negative}
|\nabla^2\rho|^2+\nabla \rho\cdot\nabla\Delta\rho+\mathrm{Ric}(\nabla\rho, \nabla\rho)
+A|\nabla\rho|^2+A\rho\Delta\rho\le 0
\quad\text{at }\theta_0.\end{equation}
Here and hereafter, we use normal coordinates $(\theta_1, \theta_2)$ at $\theta_0$. We consider two cases. 

First, assume $\rho_1=\rho_2=0$ at $\theta_0$. Then, $w\le A\bar\rho^2$ at $\theta_0$, and hence 
\begin{equation}\label{eq-case1}w\le \max\{1, A\bar\rho^2\}\quad\text{on }\Sigma.\end{equation}
Second, assume $\rho_1\neq 0$ and $\rho_2=0$ at $\theta_0$. Then, \eqref{eq-1-derivative-zero}
implies 
\begin{equation}\label{eq-relation-maximum1}\rho_{11}=-A\rho, 
\quad \rho_{12}=0\quad\text{at }\theta_0.\end{equation}
A simple evaluation of the equation \eqref{eq-LN-domain-eq-rho-new-n3} at $\theta_0$ yields 
\begin{equation}\label{eq-equation-another-form} 
\rho\rho_{22}=\frac32(\rho_1^2-1)+\big(A-\frac12\big)\rho^2\quad\text{at }\theta_0.\end{equation}
By differentiating \eqref{eq-LN-domain-eq-rho-new-n3} with respect to $\theta_1$ 
and substituting \eqref{eq-relation-maximum1}, 
we have 
\begin{equation}\label{eq-relation-maximum2}
\rho(\Delta\rho)_1=-\rho_1\rho_{22}-(2A+1)\rho\rho_1\quad\text{at }\theta_0.\end{equation}
By substituting \eqref{eq-relation-maximum1}, \eqref{eq-equation-another-form}, 
and \eqref{eq-relation-maximum2} in \eqref{eq-2-derivative-negative} and 
by a straightforward computation, we obtain
$$\big[\rho_1^2-1+\frac13(4A-1)\rho^2\big]\big[\rho_1^2-3+(2A-1)\rho^2\big]\le 0 \quad\text{at }\theta_0.$$
By a simple rearrangement, we get 
$$\big[\rho_1^2+A\rho^2-\frac13[3+(1-A)\rho^2]\big]\big[\rho_1^2+A\rho^2-[3+(1-A)\rho^2]\big]\le 0 \quad\text{at }\theta_0.$$
Therefore, 
\begin{equation}\label{eq-case2-pre}\rho_1^2+A\rho^2\le 3+(1-A)\rho^2\quad\text{at }\theta_0,\end{equation} 
and hence 
\begin{equation}\label{eq-case2}w\le 3+(1-A)\bar\rho^2\quad\text{on }\Sigma.\end{equation}

By combining \eqref{eq-case1} and \eqref{eq-case2}, we obtain 
$$w\le \max\{A\bar\rho^2, 3+(1-A)\bar\rho^2\}\quad\text{on }\Sigma.$$ 
We now take $A$ such that 
$A\bar\rho^2=3+(1-A)\bar\rho^2.$ 
Then, 
$$A=\frac{3+\bar\rho^2}{2\bar\rho^2},$$ 
and hence 
$$w\le\frac{3+\bar\rho^2}{2}\quad\text{on }\Sigma.$$
This is \eqref{eq-gradient-claim1}. 

{\it Step 2.} We claim 
\begin{equation}\label{eq-gradient-claim2}
|\nabla\rho|\le \frac{3+\bar\rho^2}{\sqrt{3(1+\bar\rho^2)}}\quad\text{on }\Sigma.\end{equation}
Consider $w=|\nabla\rho|^2$; namely, we take $A=0$ in 
\eqref{eq-definition-w}. Assume $w$ attains its maximum at some $\theta_0\in\Sigma$. 
By repeating arguments in Step 1, we obtain, instead of \eqref{eq-case2-pre}, 
\begin{equation}\label{eq-case2-pre1}\rho_1^2\le 3+\rho^2\quad\text{at }\theta_0.\end{equation} 
Evaluating \eqref{eq-gradient-claim1} at $\theta_0$, we have 
\begin{equation}\label{eq-gradient-claim1-post}
\rho_1^2+\frac{3+\bar\rho^2}{2\bar\rho^2}\rho^2\le \frac{3+\bar\rho^2}{2}\quad\text{at }\theta_0.\end{equation}
Adding the $\frac{3+\bar\rho^2}{2\bar\rho^2}$ multiple of \eqref{eq-case2-pre1} to \eqref{eq-gradient-claim1-post}, 
we get 
\begin{equation*}
\rho_1^2\le \frac{(3+\bar\rho^2)^2}{3(1+\bar\rho^2)}\quad\text{at }\theta_0.\end{equation*}
This implies \eqref{eq-gradient-claim2}. 

With $\bar\rho=4.2$, the expression in the right-hand side of \eqref{eq-gradient-claim2} is less than 2.8.
\end{proof}

The bound $2.8$ in Lemma \ref{lemma-n3-gradient}  is by no means optimal, 
but is sufficient for applications later on. 
We also point out that the universal gradient 
bound is a special property for 2-dimensional spherical domains and do not hold for higher dimensions. 


In the next two results, we improve gradient bounds under strengthened 
assumptions on $\Sigma$ for arbitrary dimensions. 

\begin{lemma}\label{lemma-C1alpha-gradient}
Let $\Sigma\subsetneq\mathbb{S}^{n-1}$ be a  
$C^{1,\alpha}$-domain, for some $\alpha\in (0,1)$, and $\rho\in C^\infty(\Sigma)\cap \mathrm{Lip}(\bar\Sigma)$ 
be the positive solution of  
\eqref{eq-LN-domain-eq-rho-new}-\eqref{eq-LN-domain-rho-boundary-new} in $\Sigma$. 
Then, $\rho\in C^{1,\alpha}(\bar\Sigma)$, and $|\nabla_\theta\rho|=1$ on $\partial\Sigma$. 
\end{lemma}

\begin{proof}
Let $d$ be the distance function in $\Sigma$ to $\partial\Sigma$. 
Then, $d$ is $C^{1,\alpha}$ near $\partial\Sigma$. 
By \eqref{eq-C1alpha-spherical}, we have 
$$|\rho-d|\le Cd^{1+\alpha}.$$
Take a function $\eta\in C^{1,\alpha}(\bar\Sigma)\cap C^2(\Sigma)$ such that 
$\eta>0$ in $\Sigma$, $\eta=0$ and $|\nabla_\theta\eta|=1$ on $\partial\Sigma$, and 
\begin{equation}\label{eq-second-derivatives}
\eta^{1-\alpha}|\nabla^2_\theta \eta|\le C\quad\text{in }\Sigma,\end{equation}
where $C$ is a positive constant depending only on $\Sigma$ and $\alpha$. 
Then, 
$$c_1\le \frac{\eta}{d}\le c_2, \quad |d-\eta|\le C\eta^{1+\alpha}\quad\text{in }\Sigma,$$ 
and hence
$$|\rho-\eta|\le C\eta^{1+\alpha}\quad\text{in }\Sigma.$$
With \eqref{eq-LN-domain-eq-rho-new}, a simple computation yields 
\begin{align*} \Delta_\theta(\rho-\eta)+{\mathbf b}\cdot \nabla_\theta(\rho-\eta)+S(\rho-\eta)
&=f\quad\text{in }\Sigma,\\
\rho-\eta&=0\quad\text{on }\partial\Sigma,
\end{align*}
where 
$${\mathbf b}=-\frac{n}{2\rho}\nabla_\theta(\rho+\eta),$$
and 
$$f=\frac{n}{2\rho}(|\nabla_\theta\eta|^2-1)-\Delta_\theta\eta-S\eta.$$
It is easy to see that $\eta {\mathbf b}$ and $\eta^2S$ are bounded in $\Sigma$ 
and that, by \eqref{eq-second-derivatives},  
$$|\eta^2f|\le C\eta^{1+\alpha}\quad\text{in }\Sigma.$$
For any $\theta\in \Sigma$, consider $B_{d(\theta)/2}(\theta)\subset \Sigma$. 
By the interior $C^{1,\alpha}$-estimate, we have 
\begin{align*}&\eta(\theta)|\nabla_\theta(\rho-\eta)(\theta)|
+\eta(\theta)^{1+\alpha}[\nabla_\theta(\rho-\eta)]_{C^\alpha(B_{d(\theta)/4}(\theta))}\\
&\qquad\le C\big\{|\rho-\eta|_{L^\infty(B_{d(\theta)/2}(\theta))}
+\eta(\theta)^2|f|_{L^\infty(B_{d(\theta)/2}(\theta))}\big\}\le C\eta(\theta)^{1+\alpha}.\end{align*}
Hence, 
$$|\nabla_\theta(\rho-\eta)|\le C\eta^{\alpha}\quad\text{in }\Sigma,$$ 
and, for any $\theta\in \Sigma$, 
\begin{align*}[\nabla_\theta(\rho-\eta)]_{C^\alpha(B_{d(\theta)/4}(\theta))}\le C.\end{align*}
Since $\nabla_\theta\eta\in C^\alpha(\bar\Sigma)$ with $|\nabla_\theta\eta|=1$ on $\partial\Sigma$, 
we conclude $\nabla_\theta\rho\in C^\alpha(\bar\Sigma)$ with $|\nabla_\theta\rho|=1$ on $\partial\Sigma$. 
\end{proof} 

For the next result, we denote by $H_{\partial\Sigma}$ the mean curvature of $\partial\Sigma$ 
with respect to the inner unit normal vector. 

\begin{lemma}\label{lemma-Lip-negative-Laplace}
Let $\Sigma\subsetneq\mathbb{S}^{n-1}$ be a  
Lipschitz domain and $\rho\in C^\infty(\Sigma)\cap \mathrm{Lip}(\bar\Sigma)$ 
be the positive solution of  
\eqref{eq-LN-domain-eq-rho-new}-\eqref{eq-LN-domain-rho-boundary-new} of $\Sigma$. 
Assume that $\Sigma$ is the union of an increasing sequence of $C^3$-domains 
$\{\Sigma_i\}$ such that $H_{\partial\Sigma_i}\ge -m$ for some nonnegative constant $m$.  Then, 
$|\nabla_\theta\rho|\le 1+c\rho$ in $\Sigma$, for some constant $c\ge 0$. 
\end{lemma} 

\begin{proof}
We first consider a special case that $\Sigma$ is a $C^3$-bounded domain such that 
$H_{\partial\Sigma}\ge -m$ for some nonnegative constant $m$. The Bochner identity yields 
$$\frac12\Delta_\theta(|\nabla_\theta\rho|^2)=|\nabla_\theta^2\rho|^2+\nabla_\theta\rho\cdot\nabla_\theta\Delta_\theta\rho
+\mathrm{Ric}(\nabla_\theta\rho, \nabla_\theta\rho).$$
On $\mathbb S^{n-1}$, $R_{ij}=(n-2)g_{ij}$. Hence, 
$$\frac12\Delta_\theta(|\nabla_\theta\rho|^2)=|\nabla_\theta^2\rho|^2+\nabla_\theta\rho\cdot\nabla_\theta\Delta_\theta\rho
+(n-2)|\nabla_\theta\rho|^2.$$
We write \eqref{eq-LN-domain-eq-rho-new} as 
$$\rho(\Delta_\theta \rho+S\rho)=\frac{n}{2}(|\nabla_\theta \rho|^2-1).$$
By applying  the Laplacian operator, we get
\begin{align*}&\rho\Delta_\theta(\Delta_\theta \rho+S\rho)
+2\nabla_\theta\rho\cdot\nabla_\theta(\Delta_\theta \rho+S\rho)
+\Delta_\theta\rho(\Delta_\theta \rho+S\rho)\\
&\qquad =
n\big(|\nabla_\theta^2\rho|^2+\nabla_\theta\rho\cdot\nabla_\theta\Delta_\theta\rho
+2S|\nabla_\theta\rho|^2\big).\end{align*}
For the last term in the left-hand side and the last two terms in the right-hand side, we write 
\begin{align*}
\Delta_\theta\rho(\Delta_\theta \rho+S\rho)
&=(\Delta_\theta\rho)^2+S\rho\Delta_\theta \rho\\
&=(\Delta_\theta\rho)^2+S\big[\frac{n}{2}|\nabla_\theta \rho|^2-S\rho^2-\frac{n}{2}\big],
\end{align*}
and 
$$\nabla_\theta\rho\cdot\nabla_\theta\Delta_\theta\rho
+2S|\nabla_\theta\rho|^2
=\nabla_\theta\rho\cdot\nabla_\theta\big(\Delta_\theta \rho+S\rho\big)
+S|\nabla_\theta\rho|^2.$$
By simple substitutions and rearrangements, we have 
\begin{align*}&\rho\Delta_\theta(\Delta_\theta \rho+S\rho)
-2S\nabla_\theta \rho \nabla_\theta\cdot(\Delta_\theta \rho+S\rho)\\
&\qquad= n|\nabla_\theta^2\rho|^2-(\Delta_\theta \rho)^2 
+\frac12S\big[n|\nabla_\theta\rho|^2+2S\rho^2+n\big]\ge 0.\end{align*}
Since $\Sigma$ has a $C^3$-boundary, we have 
\begin{equation*}\rho=d-\frac{1}{2(n-1)}H_{\partial\Sigma}d^2+O(d^3),\end{equation*}
and hence
$$\Delta_\theta \rho=-H_{\partial\Sigma}-\frac{1}{n-1}H_{\partial\Sigma}+O(d)
=-\frac{n}{n-1}H_{\partial\Sigma}+O(d).$$
By the assumption $H_{\partial \Sigma}\ge -m$ on $\partial\Sigma$, we have 
$$\Delta_\theta \rho+S\rho\le \frac{mn}{n-1}\quad\text{on }\partial \Sigma.$$
By the strong maximum principle, we obtain 
$$\Delta_\theta \rho+S\rho< \frac{mn}{n-1}\quad\text{in }\Sigma.$$
Therefore, by \eqref{eq-LN-domain-eq-rho-new}, 
$$|\nabla_\theta \rho|^2< 1+\frac{2m}{n-1}\rho\quad\text{in }\Sigma.$$

We next consider the general case that $\Sigma$ is the union of an increasing sequence of $C^3$-domains 
$\{\Sigma_i\}$ with $H_{\partial\Sigma_i}\ge -m$. Let $\rho_i$ be the positive solution of 
\eqref{eq-LN-domain-eq-rho-new}-\eqref{eq-LN-domain-rho-boundary-new} in $\Sigma_i$. By what 
we just proved in the special case, we have 
$$|\nabla_\theta\rho_i|^2\le 1+\frac{2m}{n-1}\rho\quad\text{in }\Sigma_i.$$
Note that $\rho_i\to \rho$ in $C^1(\Sigma')$ for any subdomain $\Sigma'\subset\subset\Sigma$. 
The desired result follows by a simple approximation. 
\end{proof} 

The proof of Lemma \ref{lemma-Lip-negative-Laplace} is modified from \cite{Han-Shen-20??Negative}. 

\section{Spherical Elliptic Operators}\label{sec-spherical-operators}

In this section, we discuss a class of elliptic operators over spherical domains 
with a certain singularity on the boundary. 
Our main focus is the regularity of solutions near boundary. 
Throughout this section, differentiation is always with respect to $\theta\in\mathbb{S}^{n-1}$.

Let $\Sigma\subsetneq\mathbb{S}^{n-1}$ be a 
Lipschitz domain and $d$ be the distance function in $\Sigma$ to $\partial \Sigma$. It is well-known that 
$d$ may not be $C^1$ even when $\Sigma$ has a smooth boundary. 
Take a function $\rho\in C^\infty(\Sigma)\cap C(\bar \Sigma)$ such that 
\begin{equation}\label{eq-estimate-rho-upper-lower} 
c_1\le \frac{\rho}{d}\le c_2\quad\text{on }\Sigma,\end{equation}
for some positive constants $c_1$ and $c_2$. In particular, 
$\rho$ is  positive  in $\Sigma$ 
and $\rho=0$ on $\partial\Sigma$.

For a given positive constant $\kappa$, set, for any $u\in C^2(\Sigma)$, 
\begin{equation}\label{eq-def-L-Sigma}
Lu=\Delta_\theta u-\frac{\kappa}{\rho^2}u.\end{equation}
This is a linear operator on $\Sigma$ with a singular coefficient on $\partial\Sigma$, 
since $\rho=0$ on $\partial\Sigma$. 

We first establish the existence of weak solutions of  $-Lu=f$. 

\begin{lemma}\label{lemma-singular-existnece-weak}
Let $\Sigma\subsetneq\mathbb{S}^{n-1}$ be a 
Lipschitz domain and $\rho\in C^\infty(\Sigma)\cap \mathrm{Lip}(\bar\Sigma)$ be a positive function in $\Sigma$ 
satisfying \eqref{eq-estimate-rho-upper-lower}. 
Then, for any $f\in L^2(\Sigma)$, there exists a unique weak solution $u\in H^1_0(\Sigma)$ 
of $Lu=-f$, and 
$$\|\nabla_\theta u\|_{L^2(\Sigma)}+\|\rho^{-1}u\|_{L^2(\Sigma)}\le C\|f\|_{L^2(\Sigma)}, $$
where $C$ is a positive constant depending only on $n$ and $\Sigma$.
Moreover, $u\in C^\infty(\Sigma)$. \end{lemma}

\begin{proof} By Hardy's inequality and \eqref{eq-estimate-rho-upper-lower}, we have, for any $u\in H_0^1(\Sigma)$, 
$$\|\rho^{-1}u\|_{L^2(\Sigma)}\le C\|\nabla_\theta u\|_{L^2(\Sigma)},$$ 
where $C$ is a positive constant depending only on $n$ and $\Sigma$. 
Set, for any $u, v\in H_0^1(\Sigma)$, 
\begin{align*}
(u,v)_\rho=\int_{\Omega}\big(\nabla_\theta u\cdot\nabla_\theta v
+\kappa\rho^{-2}uv\big)d\theta.\end{align*}
Denote by $\|\cdot\|_\rho$ the induced norm, i.e., for any $u\in H_0^1(\Sigma)$, 
\begin{align*}
\|u\|_\rho=\Big(\int_{\Sigma}\big(|\nabla_\theta u|^2+\kappa\rho^{-2}u^2\big)d\theta\Big)^{1/2}.\end{align*}
Then, $\|\cdot\|_\rho$ is equivalent to the standard norm in $H^1_0(\Sigma)$ and hence, 
$(H_0^1(\Sigma), (\cdot, \cdot)_{\rho})$ is a Hilbert space.
We can apply the Riesz representation theorem to conclude 
the desired result. 
\end{proof}

With the help of the compact embedding 
of $H_0^1(\Sigma)$ in $L^2(\Sigma)$, we have the following result 
concerning the eigenvalue problem for  $L$.

\begin{theorem}\label{thrm-singular-eigenvalue}
Let $\Sigma\subsetneq\mathbb{S}^{n-1}$ be a 
Lipschitz domain and $\rho\in C^\infty(\Sigma)\cap \mathrm{Lip}(\bar\Sigma)$ be a positive function in $\Sigma$ 
satisfying \eqref{eq-estimate-rho-upper-lower}. 
Then, there exist an increasing sequence of positive constants $\{\lambda_i\}_{i\ge 1}$, divergent to $\infty$, 
and an $L^2(\Sigma)$-orthonormal basis  
$\{\phi_i\}_{i\ge 1}$ 
such that, for $i\ge 1$,  $\phi_i\in C^\infty(\Sigma)\cap H_0^1(\Sigma)$ and 
$L \phi_i=-\lambda_i\phi_i$ weakly. 
\end{theorem} 

We point out that  the first eigenvalue $\lambda_1$ is of multiplicity 1. 
We also have the following Fredholm alternative. 

\begin{theorem}\label{thrm-singular-Fredholm}
Let $\Sigma\subsetneq\mathbb{S}^{n-1}$ be a 
Lipschitz domain and $\rho\in C^\infty(\Sigma)\cap \mathrm{Lip}(\bar\Sigma)$ 
be a positive function in $\Sigma$ 
satisfying \eqref{eq-estimate-rho-upper-lower}. 
Assume $\{\lambda_i\}_{i\ge 1}$ and
$\{\phi_i\}_{i\ge 1}$ are sequences of eigenvalues 
and eigenfunctions as in Theorem \ref{thrm-singular-eigenvalue}, 
respectively. 

$\mathrm{(i)}$ For any $\lambda\notin\{\lambda_i\}$ and any $f\in L^2(\Sigma)$, there exists a 
unique weak solution $u\in H^1_0(\Sigma)$ of the equation 
\begin{equation}\label{eq-weak-equation}
Lu+\lambda u=f\quad\text{in }\Sigma.\end{equation} 
Moreover, 
\begin{equation}\label{eq-H1-estimate}\|u\|_{H^1_0(\Sigma)}\le C\|f\|_{L^2(\Sigma)},\end{equation}
where $C$ is a positive constant depending only on $n$, $\lambda$, and $\Sigma$. 

$\mathrm{(ii)}$ For any $\lambda=\lambda_i$ for some $i$ 
and any $f\in L^2(\Sigma)$ with $(f, \phi_k)_{L^2(\Sigma)}=0$ 
for any eigenfunction $\phi_k$ corresponding to the eigenvalue $\lambda_i$, there exists a unique 
weak solution $u\in H^1_0(\Sigma)$ of the equation \eqref{eq-weak-equation}, satisfying 
$(u, \phi_k)_{L^2(\Sigma)}=0$ 
for any eigenfunction $\phi_k$ corresponding to the eigenvalue $\lambda_i$. Moreover, 
\eqref{eq-H1-estimate} holds. 
\end{theorem} 

Next, we study boundary regularity  of eigenfunctions $\phi_i$ in 
Theorem \ref{thrm-singular-eigenvalue} and solutions $u$ 
in Theorem \ref{thrm-singular-Fredholm}. 
In the following, we consider only a special function $\rho$, 
the unique positive solution of \eqref{eq-LN-domain-eq-rho-new}-\eqref{eq-LN-domain-rho-boundary-new},
and focus on the linear operator $L$ defined by \eqref{eq-def-L-Sigma} 
for such  $\rho$. 


We prove a more general result for later purposes. 

\begin{theorem}\label{thrm-regularity-boundary-Lip-pre}
Let $\Sigma\subsetneq\mathbb{S}^{n-1}$ be a 
Lipschitz domain, 
and $\rho\in C^\infty(\Sigma)\cap \mathrm{Lip}(\bar\Sigma)$ be the positive solution of 
\eqref{eq-LN-domain-eq-rho-new}-\eqref{eq-LN-domain-rho-boundary-new}. 
Assume that for some $\lambda\in \mathbb R$ and $f\in C^\infty(\Sigma)$, 
$u\in H^1_0(\Sigma)$ is a weak solution of \eqref{eq-weak-equation}. 
Then,  there exists a constant $\nu>0$ depending only on $n$ and $\Sigma$ such that 
if, for some constant $A>0$, 
\begin{equation}\label{eq-assumption-growth-f}
|f|\le A\rho^{\nu-2}\quad\text{in }\Sigma,\end{equation}
then
\begin{equation}\label{eq-estimate-decay-boundary-pre}|u|
\le C\big(\|u\|_{L^2(\Sigma)}+A)\rho^{\nu}\quad\text{in }\Sigma,\end{equation}
where 
$C$ is a positive constant depending only on $n$, $\kappa$, $\lambda$, and $\Sigma$.  
Moreover, $u\in C^\nu(\bar\Sigma)$ if $\nu<1$ and 
$u\in \mathrm{Lip}(\bar\Sigma)$ if $\nu\ge1$.
In particular, let $\phi_i$ be an eigenfunction as in Theorem \ref{thrm-singular-eigenvalue}.  Then, 
\begin{equation}\label{eq-eigenfunctions-Lip}|\phi_i|\le C\rho^{\nu}\quad\text{in }\Sigma,\end{equation}
where $C$ is a positive constant depending only on $n$, $\kappa$, $\lambda_i$, and $\Sigma$.  
For $n=3$, if $\kappa\ge 15/4$, then $\nu$ can be taken to satisfy 
$\nu>1/2$. 
\end{theorem}

\begin{proof} We write \eqref{eq-weak-equation} as 
\begin{equation}\label{eq-linear-operator-lambda} 
L_\lambda u\equiv\Delta_\theta u-\frac{1}{\rho^2}(\kappa-\lambda\rho^2)u=f\quad\text{in }\Sigma.\end{equation}
Take a small positive constant $\rho_0$ such that 
$\kappa>\lambda\rho_0^2$  and set 
$$\Sigma_0\equiv\{0<\rho<\rho_0\}.$$ 
For some domain $\Sigma'$ with $\Sigma\setminus\Sigma_0\subset\subset \Sigma'\subset\subset\Sigma$,
by the interior $L^\infty$-estimates, we have
\begin{equation}\label{eq-interior-estimate-pre}
\sup_{\Sigma\setminus\Sigma_0}|u|\le C\big\{\|u\|_{L^2(\Sigma')}+\|f\|_{L^\infty(\Sigma')}\big\}
\le C\big\{\|u\|_{L^2(\Sigma)}+A\big\},\end{equation}
where $C$ is a positive constant depending only on $n$, $\kappa$, $\lambda$, 
$\Sigma_0$, $\Sigma'$, and $\Sigma$.  

Take any positive constant $\nu$. By a straightforward computation 
and \eqref{eq-LN-domain-eq-rho-new}, we have 
\begin{align*}
L_\lambda\rho^\nu&=\rho^{\nu-2}\big[\nu(\nu-1)|\nabla_\theta\rho|^2-\kappa
+\nu\rho\Delta_\theta\rho+\lambda\rho^2\big]\\
&=\rho^{\nu-2}\Big[\nu\big(\nu-1+\frac{n}{2}\big)|\nabla_\theta\rho|^2-\kappa-\frac12{n\nu}
+(\lambda-\nu S)\rho^2\Big].
\end{align*}
Note that the coefficient of $|\nabla_\theta\rho|^2$ is positive. 
Since $\rho$ is Lipschitz in $\Sigma$, we fix a positive constant $\nu$ such that 
\begin{equation}\label{eq-optimal-requirement-nu}
\nu\big(\nu-1+\frac12n\big)\sup_{\{0<\rho<\rho_0\}}|\nabla_\theta\rho|^2
-\frac12n\nu<\kappa.\end{equation}
By taking $\rho_0$ sufficiently small, we have 
\begin{align*}
L_\lambda\rho^\nu\le-c_0\kappa\rho^{\nu-2}\quad\text{in }\Sigma_0,\end{align*}
for some small positive constant $c_0$. 
In the following, we take 
$$w=M\rho^\nu.$$
For $M$ sufficiently large, we have 
$$L_\lambda w\le L_\lambda u\quad\text{in }\Sigma_0.$$
By \eqref{eq-interior-estimate-pre}, 
we choose $M>0$ large further such that 
$u<w$ on $\partial\Sigma_0\cap\Sigma$.

We first discuss the case that $u$ is continuous in $\bar\Sigma$ with $u=0$ on $\partial\Sigma$. 
Note that $u=w=0$ on $\partial\Sigma$ and $u<w$ on $\partial\Sigma_0\cap\Sigma$. 
By the maximum principle, we have $u\le w$ in $\Sigma_0$. Similarly, we have 
$u\ge -w$ in $\Sigma_0$, and hence $|u|\le w$ in $\Sigma_0$. 
We obtain the desired estimate of $u$ by combining with \eqref{eq-interior-estimate-pre}. 
We note that $M=C\big(\|u\|_{L^2(\Sigma)}+A)$ for sufficiently large $C$. 

We now consider the general case that $u\in H_0^1(\Sigma)$. We note that 
$\rho^\nu$ may not be in $H^1(\Omega)$ if $\nu$ is small. 
We take $p$ large so that $[(u-w)^+]^p\in H^1_0(\Sigma_0)$. 
By multiplying $L_\lambda(u-w)\ge 0$ by $[(u-w)^+]^{2p-1}$ and integrating by parts, we have 
\begin{equation*}
\int_{\Sigma_0}\Big(\frac{2p-1}{p^2}\big|\nabla_\theta[(u-w)^+]^{p}\big|^2
+\rho^{-2}(\kappa-\lambda\rho^2) [(u-w)^+]^{2p}\Big)d\theta\le 0.\end{equation*}
We point out that all terms in the above integral makes sense. 
Then, $(u-w)^+=0$ in $\Sigma_0$, and hence $u\le w$ in $\Sigma_0$. Similarly, we have 
$u\ge -w$ in $\Sigma_0$, and hence $|u|\le w$ in $\Sigma_0$. 
We obtain the desired estimate of $u$ by combining with \eqref{eq-interior-estimate-pre}. 

For any $\theta\in \Sigma$, consider $B_{d(\theta)/2}(\theta)\subset \Sigma$. 
By the interior $C^{1}$-estimate, we have, for any $\alpha\in (0,1)$ and with $d=d(\theta)$, 
\begin{align*}d^{\alpha}[\rho]_{C^\alpha(B_{d/4}(\theta))}
+d|\nabla_\theta\rho(\theta)|
\le C\big\{|\rho|_{L^\infty(B_{d/2}(\theta))}
+d^2|f|_{L^\infty(B_{d/2}(\theta))}\big\}\le Cd^{\nu}.\end{align*}
If $\nu\in (0,1)$, we take $\alpha=\nu$ and get $[\rho]_{C^\nu(B_{d(\theta)/4}(\theta))}\le C$. Hence, 
$\rho\in C^\nu(\bar\Sigma)$, with the help of \eqref{eq-estimate-decay-boundary-pre}. 
If $\nu\ge 1$, we have $|\nabla_\theta\rho|\le C$. 
Hence, $\rho\in\mathrm{Lip}(\Sigma)$.

For $n=3$, in view of Lemma \ref{lemma-n3-gradient}, \eqref{eq-optimal-requirement-nu} reduces to 
\begin{equation}\label{eq-optimal-requirement-nu-n3}
(2.8)^2\nu\big(\nu+\frac12\big)
-\frac32\nu<\kappa.\end{equation}
If $\kappa\ge15/4$, we can take
some $\nu>1/2$ satisfying \eqref{eq-optimal-requirement-nu-n3}. 
\end{proof} 

We now make several remarks. 
First, if $\Sigma$ is a $C^2$-domain and $\rho\in C^2(\bar\Sigma)$, 
it is not necessary to assume that 
$\rho$ is a solution of 
\eqref{eq-LN-domain-eq-rho-new}-\eqref{eq-LN-domain-rho-boundary-new}, 
since there is no need to substitute 
$\Delta_\theta \rho$, which is already bounded. 
Second, the function $f$ in Theorem \ref{thrm-regularity-boundary-Lip-pre} 
may not be bounded nor $L^2$ in $\Sigma$. This is clear from the assumption \eqref{eq-assumption-growth-f}, 
if $\nu<2$. It is important for later applications that the estimate \eqref{eq-estimate-decay-boundary-pre}
does not depend on the integrals of $f$. 
Third, the assertion $\nu>1/2$ for $n=3$ plays an important role and allows us to improve 
the integrability of some singular terms to desired levels. Refer to Corollary \ref{cor-F(v)-integrability}.

In Theorem \ref{thrm-regularity-boundary-Lip-pre}, we proved that weak solutions 
are continuous up to boundary. In the next result, we prove the converse; namely, solutions 
continuous up to boundary are weak solutions. 

\begin{cor}\label{cor-regularity-boundary-Lip-weak-pre}
Let $\Sigma\subsetneq\mathbb{S}^{n-1}$ be a 
Lipschitz domain, 
and $\rho\in C^\infty(\Sigma)\cap \mathrm{Lip}(\bar\Sigma)$ be the positive solution of 
\eqref{eq-LN-domain-eq-rho-new}-\eqref{eq-LN-domain-rho-boundary-new}. 
Assume that for some  $f\in C^\infty(\Sigma)\cap L^2(\Sigma)$ satisfying 
\eqref{eq-assumption-growth-f}, 
$u\in C(\bar\Sigma)\cap C^\infty(\Sigma)$ is a solution of $-Lu=f$ in $\Sigma$ 
and $u=0$ on $\partial\Sigma$. Then, $u\in H_0^1(\Sigma)$. 
\end{cor} 

\begin{proof} By Lemma \ref{lemma-singular-existnece-weak}, 
there exists a unique weak solution $w\in H^1_0(\Sigma)$ 
of $Lw=-f$. By Theorem \ref{thrm-regularity-boundary-Lip-pre}, 
we have $w\in C(\bar\Sigma)$ with $w=0$ on $\partial\Sigma$. 
Then, the maximum principle implies that $u=w$ and hence $u\in H_0^1(\Sigma)$. 
\end{proof} 

In Corollary \ref{cor-regularity-boundary-Lip-weak-pre}, $f$ is not necessarily bounded in $\Sigma$. 

\smallskip 


Next, we improve Theorem \ref{thrm-regularity-boundary-Lip-pre}. 
The constant $\nu$ in \eqref{eq-optimal-requirement-nu} is small in general, even with an explicit lower bound 
for $n=3$.
In order to improve $\nu$, we need to get 
a more precise estimate of $|\nabla_\theta\rho|^2$ near the boundary $\partial\Sigma$. 
If $|\nabla_\theta\rho|^2$ is close to $1$ near $\partial\Sigma$, 
then we can take $\nu$ such that
$$\nu(\nu-1)<\kappa.$$
The corresponding equality will provide the optimal $\nu$. 

\begin{theorem}\label{thrm-regularity-boundary-Lip}
Let $\Sigma\subsetneq\mathbb{S}^{n-1}$ be a 
Lipschitz domain, 
and $\rho\in C^\infty(\Sigma)\cap \mathrm{Lip}(\bar\Sigma)$ be the positive solution of 
\eqref{eq-LN-domain-eq-rho-new}-\eqref{eq-LN-domain-rho-boundary-new} satisfying
\begin{equation}\label{eq-condition-gradient}|\nabla_\theta \rho|\le 1+c_0\rho^\alpha\quad\text{in }\Sigma,
\end{equation}
for some constants $c_0\ge 0$ and $\alpha\in (0,1]$. 
Assume that $s$ is the positive constant satisfying 
\begin{equation}\label{eq-definition-s}s(s-1)=\kappa,\end{equation} and that, 
for some $\lambda\in \mathbb R$ and $f\in C^\infty(\Sigma)$, 
$u\in H^1_0(\Sigma)$ is a weak solution of \eqref{eq-weak-equation}. 
If, for some $A>0$ and $a>0$ with $a\neq s$, 
$$|f|\le A\rho^{a-2}\quad\text{in }\Sigma,$$
then 
\begin{equation}\label{eq-estimate-decay-boundary}|u|+\rho|\nabla_\theta u|\le 
C\big(\|u\|_{L^2(\Sigma)}+A\big)\rho^{b}\quad\text{in }\Sigma,\end{equation}
where $b=\min\{a, s\}$, and 
$C$ is a positive constant depending only on $n$, $\alpha$, $c_0$, $\kappa$, $a$, $\lambda$, and $\Sigma$.  
Moreover, $u\in C^b(\bar\Sigma)$ if $b<1$ and $u\in \mathrm{Lip}(\bar\Sigma)$ if $b\ge 1$. 
In particular, let $\phi_i$ be an eigenfunction as in Theorem \ref{thrm-singular-eigenvalue}.  Then, 
$\phi_i\in \mathrm{Lip}(\bar\Sigma)$ and 
$$|\phi_i|+\rho|\nabla_\theta\phi_i|\le C\rho^{s}\quad\text{in }\Sigma,$$
where $C$ is a positive constant depending only on $n$, $\kappa$, $\alpha$, $c_0$, $\lambda_i$, and $\Sigma$.  
\end{theorem}

\begin{proof} 
We will prove the estimate of $u$ itself in \eqref{eq-estimate-decay-boundary}. 
The estimate of the gradient $\nabla_\theta u$ follows from the interior $C^{1}$-estimate, as 
in the proof of Theorem \ref{thrm-regularity-boundary-Lip-pre}. 
Let $L_\lambda$ be the operator in \eqref{eq-linear-operator-lambda}. 
By Theorem \ref{thrm-regularity-boundary-Lip-pre}, $u\in C^\nu(\bar\Sigma)$ with $u=0$ on $\partial\Sigma$, 
for some $\nu>0$. 

We modify the proof of Theorem \ref{thrm-regularity-boundary-Lip-pre} 
and construct appropriate supersolutions. 
Take any constant $a$. By a straightforward computation and \eqref{eq-LN-domain-eq-rho-new}, we have 
\begin{align*}
L_\lambda\rho^a
=\rho^{a-2}\Big[a\big(a-1+\frac{n}{2}\big)|\nabla_\theta\rho|^2-\kappa-\frac{na}{2}
+(\lambda-aS)\rho^2\Big].
\end{align*}
Note that the coefficient of $|\nabla_\theta\rho|^2$ is positive. 
By  \eqref{eq-condition-gradient}, we get 
\begin{align*}
L_\lambda\rho^a
&\le\rho^{a-2}\Big[a\big(a-1+\frac{n}{2}\big)(1+c_0\rho^\alpha)-\kappa-\frac{na}{2}+(\lambda-aS)\rho^2\Big]\\
&=\rho^{a-2}\Big[a(a-1)-\kappa+c_0a\big(a-1+\frac{n}{2}\big)\rho^\alpha+(\lambda-aS)\rho^2\Big].
\end{align*}
If $a<s$, then $a(a-1)<\kappa$. By taking $\rho_0$ sufficiently small, we have 
\begin{align*}
L_\lambda\rho^a\le-\frac12\big(\kappa-a(a-1)\big)\rho^{a-2}\quad\text{in }\Sigma_0.\end{align*}
For $a>s$, we take some fixed $\tau>s$. Then, 
\begin{align*}
L_\lambda(\rho^s-\rho^\tau)
&=\rho^{s-2}\Big[s\big(s-1+\frac{n}{2}\big)-\tau\rho^{\tau-s}\big(\tau-1+\frac{n}{2}\big)\Big]|\nabla_\theta\rho|^2\\
&\qquad+\rho^{s-2}\Big[-\kappa-\frac{ns}{2}+(\lambda-sS)\rho^2\Big]
-\rho^{\tau-2}\Big[-\kappa-\frac{n\tau}{2}+(\lambda-\tau S)\rho^2\Big].
\end{align*}
For $\rho$ sufficiently small, the coefficient of $|\nabla_\theta \rho|^2$ is nonnegative. 
By replacing $|\nabla_\theta \rho|^2$ 
by its upper bound $1+c_0\rho^\alpha$ due to \eqref{eq-condition-gradient}, we have 
\begin{align*}
L_\lambda(\rho^s-\rho^\tau)&\le 
\rho^{\tau-2}\Big[-(\tau(\tau-1)-\kappa)
+(\lambda-s S)\rho^{s+2-\tau}-(\lambda-\tau S)\rho^2\\
&\qquad +c_0\rho^{s+\alpha-\tau}\big[\big(s(s-1)+\frac{ns}{2}\big)-\rho^{\tau-s}\big(\tau(\tau-1)+\frac{n\tau}{2}\big)\big]\Big],
\end{align*}
where we used $s(s-1)=\kappa$. 
Take $\tau\in (s, s+\alpha)$ with $\tau\le a$. 
Since $\tau>s$, we have $\tau(\tau-1)>\kappa$. 
(The requirement $a\ge \tau>s$ prohibits $a=s$!) 
By taking $\rho_0$ sufficiently small, we have 
$$L_\lambda(\rho^s-\rho^\tau)\le -\frac12\big(\tau(\tau-1)-\kappa\big)\rho^{a-2}\quad\text{in }\Sigma_0.$$

In the following, we take, for $a\in (0,s)$, 
$$w=M\rho^a,$$
and, for $a>s$, 
$$w=M(\rho^s-\rho^\tau).$$
For $M$ sufficiently large, we have 
$$L_\lambda w\le -A\rho^{a-2}\le L_\lambda(\pm u)\quad\text{in }\Sigma_0.$$
By taking $\rho_0$ small further, we assume 
that $w>0$ on $\partial\Sigma_0\cap\Sigma$. By \eqref{eq-interior-estimate-pre}, 
we choose $M>0$ large further such that 
$u<w$ on $\partial\Sigma_0\cap\Sigma$. 
Note that $u=w=0$ on $\partial\Sigma$. 
By the maximum principle, we get 
$\pm u\le w$ in $\Sigma_0$, and hence $|u|\le w$ in $\Sigma_0$. 
We obtain the desired estimate of $u$ by combining with \eqref{eq-interior-estimate-pre}. 
\end{proof} 

By Lemma \ref{lemma-C1alpha-gradient} and Lemma \ref{lemma-Lip-negative-Laplace}, 
we conclude that Theorem \ref{thrm-regularity-boundary-Lip} holds for domains 
as in Lemma \ref{lemma-C1alpha-gradient} and Lemma \ref{lemma-Lip-negative-Laplace}. 
We point out that the power $b=\min\{a,s\}$ in \eqref{eq-estimate-decay-boundary} is optimal for both cases $a<s$
and $a>s$. 



\section{Estimates near Cylinderical Boundaries}\label{sec-boundary-regularity} 

In this section, we discuss the regularity of solutions of the Yamabe equation on cylinders and 
focus on the regularity near the boundary. 

Let $\Sigma\subsetneq\mathbb{S}^{n-1}$ be a Lipschitz domain, 
and $\rho\in C^\infty(\Sigma)\cap \mathrm{Lip}(\bar\Sigma)$ be the positive solution 
of \eqref{eq-LN-domain-eq-rho-new}-\eqref{eq-LN-domain-rho-boundary-new}. 
For some positive constants $\kappa$, $\beta$, and $T$, 
consider the equation 
\begin{equation}\label{eq-basic-equation-v-new}
\mathcal Lv= F(v)\quad\text{in }(T,\infty)\times\Sigma,\end{equation} 
where $\mathcal L$ and $F$ are given by 
\begin{equation}\label{eq-def-mathcal-L-new1}
\mathcal Lv=\partial_{tt}v+\Delta_\theta v-\frac{\kappa}{\rho^2}v
-\beta^2 v,\end{equation}
and, for some smooth function $h$  on $(-1,1)$ with $h(0)\neq 0$, 
\begin{equation}\label{eq-def-mathcal-F-new}
F(v)=\rho^{\beta-2}v^2
h\big(\rho^{\beta}v\big).\end{equation}
Here and hereafter, we always take 
\begin{align}\label{eq-definition-constants}
\kappa=\frac14n(n+2), \quad \beta=\frac12(n-2).\end{align} 
Then, \eqref{eq-basic-equation-v-new}, \eqref{eq-def-mathcal-L-new1}, and \eqref{eq-def-mathcal-F-new}
are \eqref{eq-basic-equation-v}, \eqref{eq-def-mathcal-L}, and \eqref{eq-def-mathcal-F}, respectively. 
We note that $\kappa=15/4$ for $n=3$. This fact is used in Theorem \ref{thrm-regularity-boundary-Lip-pre}.

For some positive constants $\gamma$ and $A$, 
let $v$ be a smooth solution of \eqref{eq-basic-equation-v-new} in $(T,\infty)\times\Sigma$ satisfying
\begin{equation}\label{eq-estimate-v-basic-gamma}|v|\le A e^{-\gamma t}\quad\text{in }(T,\infty)\times\Sigma.
\end{equation}
We note that \eqref{eq-v1-0} implies 
\eqref{eq-estimate-v-basic-gamma} for $\gamma=\beta$ and some  universal constant $A$ depending only on $n$.
In particular,  $v$ is bounded. 
Now, we derive a decay estimate near boundary. 

\begin{lemma}\label{lemma-estimate-cylinderical-boundary} 
Let $\Sigma\subsetneq\mathbb{S}^{n-1}$ be a 
Lipschitz domain, 
and $\rho\in C^\infty(\Sigma)\cap \mathrm{Lip}(\Sigma)$ be the positive solution of 
\eqref{eq-LN-domain-eq-rho-new}-\eqref{eq-LN-domain-rho-boundary-new}. 
Assume that $v$ is a smooth solution of \eqref{eq-basic-equation-v-new} in $(T,\infty)\times\Sigma$
satisfying \eqref{eq-estimate-v-basic-gamma}, for some constants $\gamma\ge \beta$ and $A>0$.  
Then, 
\begin{equation}\label{eq-estimate-cylinderical-boundary-pre} 
|v|+\rho|\nabla_{(t,\theta)}v|
\le Ce^{-\gamma t}\rho^{\nu}
\quad\text{in }(T+1,\infty)\times\Sigma,\end{equation}
where  $\nu$ is the positive constant as in Theorem \ref{thrm-regularity-boundary-Lip-pre}, and 
$C$ is a positive constant  depending only on $n$, $\gamma$, $A$, and $\Sigma$.
Moreover, $v\in C^\nu((T,\infty)\times\bar\Sigma)$ if $\nu<1$ and 
$v\in \mathrm{Lip}((T,\infty)\times\bar\Sigma)$ if $\nu\ge 1$. 
\end{lemma} 

\begin{proof} 
By \eqref{eq-def-mathcal-F-new} and \eqref{eq-estimate-v-basic-gamma}, we have 
\begin{equation}\label{eq-estimate-initial-linear-a-pre0}
|F(v)|\le Ce^{-2\gamma t}\rho^{\beta-2}
\le Ce^{-\gamma t}\rho^{\beta-2}.\end{equation}
We point out that the assumption on $v$ in \eqref{eq-estimate-v-basic-gamma} implies that $v$ is bounded. 
It is not known whether $v$ is  continuous in $(T,\infty)\times\bar\Sigma$, 
with $v=0$ on $(T,\infty)\times\partial\Sigma$. 
For a remedy, we consider 
$\rho^\mu v$, for some $\mu>0$. A straightforward computation yields 
$$\mathcal L_*(\rho^\mu v)=\rho^\mu F,$$
where 
\begin{align}\label{eq-operator-L-star0}
\mathcal L_*u=\partial_{tt}u+\Delta_\theta u
-\frac{2\mu}{\rho}\nabla_\theta\rho\cdot\nabla_\theta u+\frac{c}{\rho^2} u-\beta^2u,\end{align}
and 
\begin{align*} 
c=-\kappa+\mu\Big(-\big(\frac n2-(\mu+1)\big)|\nabla_\theta\rho|^2+\frac n2+S\rho^2\Big).
\end{align*}
Here, $S$ is the constant given by \eqref{eq-constant-S}. Fix a $t_0>T+1$, and set 
$$\Omega_0=\{(t, \theta); |t-t_0|<1, 0<\rho<\rho_0\}.$$
Note that 
\begin{align*} 
c\le -\kappa+\mu\big(\frac n2+S\rho^2\big)\le -\frac12\kappa \quad\text{in }\Omega_0,
\end{align*} 
if $\mu+1\le n/2$, $n\mu<\kappa/2$, and $\rho_0$ is small.  

Consider  
\begin{equation}\label{eq-barrier-case1-pre0}
w=e^{-\gamma t}\big[\varepsilon(t-t_0)^2+\rho^b\big].\end{equation}
A straightforward calculation, with the help of \eqref{eq-LN-domain-eq-rho-new}, yields
\begin{align*} 
\mathcal L_* w=I_1+I_2,\end{align*} 
where 
\begin{align*}
I_1=\varepsilon e^{-\gamma t}
\Big\{\big[(\gamma^2-\beta^2)(t-t_0)^2
-4\gamma(t-t_0)+2\big] +c\rho^{-2}(t-t_0)^2\Big\},\end{align*}
and 
\begin{align*}
I_2=e^{-\gamma t}\rho^{b-2}
\Big\{(\gamma^2-\beta^2)\rho^2
-\kappa+(b-\mu)\big[\big(b-\mu-1+\frac12n\big)|\nabla_\theta\rho|^2
-\frac12n-S\rho^2\big]\Big\}.
\end{align*}
Choose $b>\mu$. Then, the coefficient of $|\nabla_\theta\rho|^2$ in $I_2$ is positive. 
By $c\le 0$, we have 
\begin{align*}I_1\le C_1 \varepsilon e^{-\gamma t},\end{align*}
and 
\begin{align*}
I_2&\le e^{-\gamma t}\rho^{b-2}
\big\{C_2\rho^2-\kappa+(b-\mu)\big(b-\mu-1+\frac12n\big)\sup_{\Omega_0}|\nabla\rho|^2
-\frac12n(b-\mu)\big\}.
\end{align*}
By taking $b$, $\varepsilon$, and $\rho_0$ small, with $b\le 2$ in particular,  we obtain 
\begin{align*}\mathcal L_* w=I_1+I_2\le -c_0\kappa 
e^{-\gamma t}\rho^{b-2}\quad\text{in }\Omega_0,
\end{align*}
for some small positive constant $c_0$. By \eqref{eq-estimate-initial-linear-a-pre0}, we have 
$$|\rho^\mu F|\le  
Ce^{-\gamma t}\rho^{\beta+\mu-2}.$$
By taking $M$ large, we obtain 
$$\mathcal L_*(Mw)\le -c_0M\kappa 
e^{-\gamma t}\rho^{b-2}\le
-Ae^{-\gamma t}\rho^{\beta+\mu-2}\le \mathcal L_*(\pm \rho^\mu v)\quad\text{in }\Omega_0,$$
if $b\le \beta+\mu$. 
Next, we claim, by choosing $M$ large further if necessary, 
\begin{align}\label{eq-boundary-comparison-pre0}
\pm \rho^\mu v\le Mw \quad\text{on }\partial\Omega_0.\end{align}
First,  $\rho^\mu v=0$ on $(T,\infty)\times\partial\Sigma$. 
(Here, we used $\mu>0$.)
Next, by \eqref{eq-estimate-v-basic-gamma}, we have,  
$\pm \rho^\mu v\le Ce^{-\gamma t}$  on $|t-t_0|=1$ or $\rho=\rho_0$. 
On the other hand, $w=0$ on $(T,\infty)\times\partial\Sigma$, 
$w\ge e^{-\gamma t}$
on $|t-t_0|=1$, and 
$w\ge e^{-\gamma t}\rho_0^b$ on $\rho=\rho_0$. 
This proves \eqref{eq-boundary-comparison-pre0} for some $M$ sufficiently large. 
We emphasize that $M$ is independent of $t_0$. 
By the maximum principle, we obtain 
$\pm (\rho^\mu v)\le Mw$ in $\Omega_0$, and in particular at $t=t_0$, 
$$|v(t_0, \theta)|\le Me^{-\gamma t_0}\rho^{b-\mu}.$$
This holds for any $t_0>T+1$ and hence proves the estimate of $v$ in \eqref{eq-estimate-cylinderical-boundary-pre} 
for $\nu=b-\mu>0$. In fact, we can choose $\nu$ to satisfy $\nu\le \min\{2,\beta\}$ and 
\eqref{eq-optimal-requirement-nu}.

The assertion concerning the H\"older continuity or the Lipschitz continuity of $v$ 
follows from the interior $C^{1}$-estimate, as 
in the proof of Theorem \ref{thrm-regularity-boundary-Lip-pre}. 
\end{proof} 

If $v$ is  continuous in $(T,\infty)\times\bar\Sigma$
with $v=0$ on $(T,\infty)\times\partial\Sigma$, 
there is no need to introduce $\mathcal L_*$ in \eqref{eq-operator-L-star0}. 
We can simply apply $\mathcal L$  to $w$ in \eqref{eq-barrier-case1-pre0}. 
In other words, we can take $\mu=0$ in the proof above.

In the next result, we discuss derivatives of solutions with respect to $t$. 
As we see in the proof, it is more helpful to view 
$F(v)$ in the equation \eqref{eq-basic-equation-v-new} 
as a perturbative term of the operator $\mathcal Lv$, instead of a nonhomogeneous term. 

\begin{lemma}\label{lemma-regularity-v-derivatives-boundary} 
Let $\Sigma\subsetneq\mathbb{S}^{n-1}$ be a 
Lipschitz domain, 
and $\rho\in C^\infty(\Sigma)\cap \mathrm{Lip}(\Sigma)$ be the positive solution of 
\eqref{eq-LN-domain-eq-rho-new}-\eqref{eq-LN-domain-rho-boundary-new}. 
Assume that $v$ is a smooth solution of \eqref{eq-basic-equation-v-new} in $(T,\infty)\times\Sigma$
satisfying \eqref{eq-estimate-v-basic-gamma}, for some constants $\gamma\ge \beta$ and $A>0$. 
Then, 
\begin{equation}\label{eq-estimate-derivative-cylinderical-boundary-basic} 
|\partial_tv|+|\partial_{tt}v|\le Ce^{-\gamma t}\rho^{\nu}
\quad\text{in }(T+1,\infty)\times\Sigma,\end{equation}
where $\nu$ is the constant as in  Theorem \ref{thrm-regularity-boundary-Lip-pre}, 
and $C$ is a positive constant  depending only on $n$, $\gamma$, $A$, and $\Sigma$. 
In particular, $\partial_tv, \partial_{tt}v\in C((T+1, \infty)\times\bar\Sigma)$ with $\partial_tv=\partial_{tt}v=0$ 
on $(T+1, \infty)\times \partial\Sigma$. 
\end{lemma} 

\begin{proof} We first derive the estimate of $\partial_tv$. Set 
\begin{equation}\label{eq-definition-c}\epsilon(s)=sh(s).\end{equation}
Then, $\epsilon$ is a smooth function on $(-1,1)$ with $\epsilon(0)=0$, and 
$$F(v)=\rho^{-2}v\epsilon(\rho^\beta v).$$
We now write \eqref{eq-basic-equation-v-new} 
as 
\begin{equation}\label{eq-equation-v-homo} 
\partial_{tt}v+\Delta_\theta v-\frac{1}{\rho^2}\big(\kappa+\epsilon(\rho^\beta v)\big)v-\beta^2v=0.\end{equation}
Then, \eqref{eq-estimate-v-basic-gamma} implies 
$$|\epsilon(\rho^\beta v)|\le C\rho^\beta |v|\le C.$$
Consider any $\Sigma'\subset\subset \Sigma$. For any $t>T+1$ and $\theta\in \Sigma'$, 
by the interior $C^1$-estimates and \eqref{eq-estimate-v-basic-gamma}, we have, with $r=\min\{ d(\theta)/2,1\}$,
\begin{equation}\label{eq-estimate-der-v-interior}
|\partial_tv(t,\theta)|\le C\sup_{(t-r, t+r)\times B_{r}(\theta)}|v|\le Ce^{-\gamma t}.\end{equation}
To estimate $\partial_tv$ near $\partial\Sigma$, we fix an arbitrary $t_0>T+2$ 
and consider, for some $\rho_0>0$ to be determined, 
$$\Omega_0=\{(t,\theta);\, t_0<t< t_0+1,\, 0<\rho(\theta)<\rho_0\}.$$
Set
\begin{equation*}
v_1(t)=v(t)-v(2t_0-t).\end{equation*} 
We now evaluate the equation \eqref{eq-equation-v-homo} at $t$ and $2t_0-t$, respectively, and take a difference. 
A straightforward computation yields 
\begin{equation}\label{eq-equation-v-difference}
\partial_{tt}v_1+\Delta_\theta v_1-\frac{1}{\rho^2}\big(\kappa+\widetilde \epsilon\big)v_1-\beta^2v_1=0,\end{equation}
where $\widetilde c$ is given by 
$$\widetilde \epsilon(t, \theta)
=\epsilon(\rho^\beta v(t))+\rho^\beta v(2t_0-t)\int_0^1\epsilon'\big(s\rho^\beta v(t)+(1-s)\rho^\beta v(2t_0-t)\big)ds.$$
For any $t_0<t<t_0+1$, we have 
$$|\widetilde \epsilon(t, \theta)|
\le C\rho^\beta (|v(t)|+ |v(2t_0-t)|)\le Ce^{-\gamma t_0}\rho^\beta.$$
Set 
$$\mathcal L_1w=\partial_{tt}w+\Delta_\theta w-\frac{1}{\rho^2}\big(\kappa+\widetilde \epsilon\big)w-\beta^2w.$$ 
Then, $\mathcal L_1v_1=0$. To construct a supersolution, set 
\begin{equation}\label{eq-definition-supersolution-w}
w=(t-t_0)e^{-\gamma t}\rho^\nu.\end{equation} 
A straightforward computation yields 
\begin{align*}\mathcal L_1w&=
(t-t_0)e^{-\gamma t}\rho^{\nu-2}
\big[(\gamma^2-\beta^2)\rho^2
-\kappa+\nu(\nu-1+\frac n2)|\nabla_\theta\rho|^2-\frac12n\nu-\nu S\rho^2+\widetilde \epsilon\big]\\
&\qquad -2\gamma e^{-\gamma t}\rho^\nu.
\end{align*}
We now take $\nu$ to satisfy \eqref{eq-optimal-requirement-nu}. 
By taking $\rho_0$ sufficiently small, we have 
$$\mathcal L_1w\le -c_0\kappa (t-t_0)e^{-\gamma t}\rho^{\nu-2}\le 0\quad\text{in }\Omega_0,$$
for some small positive constant $c_0$. 
Next, we compare $\pm v_1$ and $w$ on $\partial\Omega_0$. First, $v_1=0$ on $(t_0, t_0+1)\times \partial\Sigma$ 
and on $\{t=t_0\}\times \Sigma$. Next, we get, by \eqref{eq-estimate-cylinderical-boundary-pre}, 
$$|v_1|\le Ce^{-\gamma t_0}\rho^\nu\quad\text{on }\{t_0+1\}\times\Sigma,$$
and, by \eqref{eq-estimate-der-v-interior}, 
$$|v_1|\le C(t-t_0)\sup_{(t_0-1, t_0+1)\times\{\rho=\rho_0\}}|\partial_tv|\le C(t-t_0)e^{-\gamma t}
\quad\text{on }(t_0, t_0+1)\times\{\rho=\rho_0\}.$$
For $w$, we have $w=0$ on $(t_0, t_0+1)\times \partial\Sigma$ 
and on $\{t=t_0\}\times \Sigma$, $w=e^{-\gamma}e^{-\gamma t_0}\rho^\nu$ on $\{t=t_0+1\}\times \Sigma$ and 
$w= (t-t_0)e^{-\gamma t_0}\rho_0^\nu$ on $(t_0, t_0+1)\times\{\rho=\rho_0\}$. By choosing $M$ sufficiently large, 
independent of $t_0$, we have $\pm v_1\le Mw$ on $\partial\Omega_0$. By the maximum principle, we obtain 
$\pm v_1\le Mw$ in $\Omega_0$, and hence 
$$|v(t)-v(2t_0-t)|\le M(t-t_0)e^{-\gamma t}\rho^\nu\quad\text{in }\Omega_0.$$
By dividing by $t-t_0$ and taking the limit as $t\to t_0+$, we get, for any $\rho\in (0,\rho_0)$, 
$$|\partial_t v(t_0, \theta)|\le Me^{-\gamma t_0}\rho^\nu.$$
This holds for any $t_0>T+2$. By combining with \eqref{eq-estimate-der-v-interior}
for $\Sigma'=\{\rho>\rho_0\}$, we obtain 
\begin{equation}\label{eq-estimate-v-derivative-cylinderical-boundary-basic} 
|\partial_tv|\le Ce^{-\gamma t}\rho^{\nu}
\quad\text{in }(T+2,\infty)\times\Sigma.\end{equation}
This is the estimate of $\partial_tv$ in \eqref{eq-estimate-derivative-cylinderical-boundary-basic}.

The derivation of the estimate of $\partial_{tt}v$ is similar. We point out some key steps. 
By differentiating \eqref{eq-equation-v-homo} with respect to $t$, we have 
\begin{equation}\label{eq-equation-v-deri-homo} 
\partial_{tt}v_t+\Delta_\theta v_t-\frac{1}{\rho^2}\big(\kappa+\epsilon_1(\rho^\beta v)\big)v_t-\beta^2v_t=0,\end{equation}
where 
$$\epsilon_1(s)=\epsilon(s)+s\epsilon'(s).$$
Note that $\epsilon_1$ is a smooth function on $(-1,1)$ with $\epsilon_1(0)=0$. 
Similar as \eqref{eq-estimate-der-v-interior}, 
we have, for any $t>T+3$ and any $\theta\in\Sigma'\subset\subset\Sigma$, 
\begin{equation}\label{eq-estimate-der-v-der-interior}
|\partial_{tt}v(t,\theta)|\le  Ce^{-\gamma t}.\end{equation}
Set
\begin{equation*}
v_2(t)=v_t(t)-v_t(2t_0-t).\end{equation*} 
Similarly, we have 
\begin{equation}\label{eq-equation-v-deri-difference}
\partial_{tt}v_2+\Delta_\theta v_2
-\frac{1}{\rho^2}\big(\kappa+\epsilon_1(\rho^\beta v)\big)v_2-\beta^2v_2=f_1,\end{equation}
where $f_1$ is given by 
$$f_1=\rho^{-2}\big[\epsilon_1(\rho^\beta v(t))-\epsilon_1(\rho^\beta v(2t_0-t))\big] v_t(2t_0-t).$$
By \eqref{eq-estimate-v-derivative-cylinderical-boundary-basic}, we have
\begin{align*}|f_1|&\le C\rho^{\beta-2}|v(t) -v(2t_0-t)| |v_t(2t_0-t)|\\
&\le C(t-t_0)e^{-2\gamma t}\rho^{\beta+2\nu-2}
\le C(t-t_0)e^{-\gamma t}\rho^{\nu-2}.\end{align*}
The rest of the proof is almost identical as the derivation of $\partial_tv$, and hence will be omitted. 
We point out that the equation for $v_1$ in \eqref{eq-equation-v-difference} 
can be regarded as a homogeneous equation and that 
the equation for $v_2$ in \eqref{eq-equation-v-deri-difference}
is not homogeneous. The function $w$ in \eqref{eq-definition-supersolution-w} still serves 
as a supersolution. 
\end{proof}

We can derive estimates similar as \eqref{eq-estimate-derivative-cylinderical-boundary-basic} 
for higher derivatives of $v$ with respect to $t$. However, the estimate of the second derivative 
is sufficient for our applications. 

\begin{cor}\label{cor-regularity-v-H10} 
Let $\Sigma\subsetneq\mathbb{S}^{n-1}$ be a 
Lipschitz domain, 
and $\rho\in C^\infty(\Sigma)\cap \mathrm{Lip}(\Sigma)$ be the positive solution of 
\eqref{eq-LN-domain-eq-rho-new}-\eqref{eq-LN-domain-rho-boundary-new}. 
Assume that $v$ is a smooth solution of \eqref{eq-basic-equation-v-new} in $(T,\infty)\times\Sigma$
satisfying \eqref{eq-estimate-v-basic-gamma},  for some constants $\gamma\ge \beta$ and $A>0$. 
Then, for any $t>T+1$, 
$v(t,\cdot)\in H_0^1(\Sigma)$, and 
\begin{align}\label{eq-estimate-gradient-v1}
\|\nabla_\theta v(t, \cdot)\|_{L^2(\Sigma)}+\|\rho^{-1}v(t, \cdot)\|_{L^2(\Sigma)}
\le C
e^{-\gamma t},\end{align}
where $C$ is a positive constant  depending only on $n$, $\gamma$, $A$, and $\Sigma$. 
\end{cor} 

\begin{proof} 
We write the equation \eqref{eq-basic-equation-v-new} or 
\eqref{eq-equation-v-homo} as 
\begin{equation}\label{eq-equation-v-homo-equiv} 
\Delta_\theta v-\frac{1}{\rho^2}\big(\kappa+\epsilon(\rho^\beta v)\big)v=-\partial_{tt}v+\beta^2v.\end{equation}
By $\epsilon(0)=0$ and \eqref{eq-estimate-v-basic-gamma}, we take $T_*$ large such that 
$|\epsilon(\rho^\beta v)|<\kappa/2$ for $t>T_*$. By \eqref{eq-estimate-derivative-cylinderical-boundary-basic}, 
the right-hand side of \eqref{eq-equation-v-homo-equiv} is a bounded function. 
By restricting \eqref{eq-equation-v-homo-equiv} on $\{t\}\times\Sigma$ for each $t>T_*$, we conclude 
$v(t,\cdot)\in H_0^1(\Sigma)$, by Corollary \ref{cor-regularity-boundary-Lip-weak-pre}. 
By multiplying \eqref{eq-equation-v-homo-equiv} by $-v$ and integrating over $\{t\}\times\Sigma$, 
we have  
\begin{align}\label{eq-integral-identity}
\int_{\{t\}\times\Sigma}\big[|\nabla_\theta v|^2+\big(\kappa+\epsilon(\rho^\beta v)\big)\rho^{-2}v^2\big]d\theta
= \int_{\{t\}\times\Sigma}\big(\partial_{tt}v-\beta^2v)vd\theta. 
\end{align}
By \eqref{eq-estimate-v-basic-gamma}
and \eqref{eq-estimate-derivative-cylinderical-boundary-basic}, we get 
\begin{align}\label{eq-estimate-gradient-v}
\int_{\Sigma}\big(|\nabla_\theta v(t, \cdot)|^2+\rho^{-2}v^2(t,\cdot)\big)d\theta\le C
e^{-\gamma t}\|v(t, \cdot)\|_{L^2(\Sigma)}.\end{align}
This implies the desired result. 
\end{proof}

The property 
$v(t,\cdot)\in H_0^1(\Sigma)$ plays an essential role 
and permits us to integrate by parts on each slice $\{t\}\times\Sigma$, 
in deriving \eqref{eq-integral-identity}.

\begin{cor}\label{cor-F(v)-integrability} 
Let $\Sigma\subsetneq\mathbb{S}^{n-1}$ be a 
Lipschitz domain, 
and $\rho\in C^\infty(\Sigma)\cap \mathrm{Lip}(\Sigma)$ be the positive solution of 
\eqref{eq-LN-domain-eq-rho-new}-\eqref{eq-LN-domain-rho-boundary-new}. 
Assume that $v$ is a smooth solution of \eqref{eq-basic-equation-v-new} in $(T,\infty)\times\Sigma$
satisfying \eqref{eq-estimate-v-basic-gamma}. Then for any $t>T$, 
$F(v)(t, \cdot)\in L^p(\Sigma)$ for some $p>\max\{2, n/2\}$, and, for any $q\in [1,p]$,  
\begin{equation}\label{eq-estimate-nonlinear-Lp}\|F(v)(t,\cdot)\|_{L^q(\Sigma)}
\le  Ce^{-2\gamma t},\end{equation}
where $C$ is a positive constant depending only on $n$, $\gamma$, $A$, and $\Sigma$.
\end{cor}

\begin{proof} Recall that $\beta=\frac12(n-2)$ by \eqref{eq-definition-constants}. 
By \eqref{eq-def-mathcal-F-new}, we have 
$$|F(v)|\le C\rho^{\frac{n-6}2}v^2=C\rho^{\frac{n-2}2}(\rho^{-1}v)^2.$$
Then, $F(v)(t, \cdot)\in L^1(\Sigma)$ and \eqref{eq-estimate-nonlinear-Lp} holds for $q=1$, 
by \eqref{eq-estimate-gradient-v1}.

If $n\ge 6$, then $F(v)$ is bounded and $|F(v)|\le Cv^2$. Hence, for each $t>T$, 
$$\|F(v)(t,\cdot)\|_{L^\infty(\Sigma)}\le C\|v(t,\cdot)\|^2_{L^\infty(\Sigma)}.$$
For $n=5$, we have $\beta=3/2$ and then
$$|F(v)|\le C\rho^{-\frac{1}2}v^2= Cv^{\frac{3}2}\big(\rho^{-1}v\big)^{\frac{1}2}.$$
Hence, for each $t>T$, 
$$\|F(v)(t,\cdot)\|_{L^{4}(\Sigma)}
\le  C\|v(t,\cdot)\|^{\frac{3}2}_{L^\infty(\Sigma)}\, \|\rho^{-1}v(t,\cdot)\|^{\frac{1}2}_{L^2(\Sigma)}.$$
We obtain the desired estimate for $n\ge 5$ by \eqref{eq-estimate-v-basic-gamma} and 
\eqref{eq-estimate-gradient-v1}.

For $n=4$, we have  $\beta=1$. If $\nu<1/2$,
then, 
\begin{align*}|F(v)|\le C\rho^{-1}v^2
=C(\rho^{-\nu}v)^{\frac{1}{1-\nu}}(\rho^{-1}v)^{\frac{1-2\nu}{1-\nu}},\end{align*}
and hence, by \eqref{eq-estimate-cylinderical-boundary-pre} and \eqref{eq-estimate-gradient-v1},
$$\|F(v)(t,\cdot)\|_{L^{\frac{2(1-\nu)}{1-2\nu}}(\Sigma)}
\le  C(e^{-\gamma t})^{\frac{1}{1-\nu}} \|\rho^{-1}v(t,\cdot)\|^{\frac{1-2\nu}{1-\nu}}_{L^2(\Sigma)}
\le Ce^{-2\gamma t}.$$
Note that $\frac{2(1-\nu)}{1-2\nu}>2$. If $\nu\ge 1$, we get a similar estimate of 
the $L^\infty$-norm.

For $n=3$, we have  $\beta=1/2$. 
If $\nu<3/4$, 
then, 
\begin{align*}|F(v)|\le C\rho^{-\frac32}v^2
=C(\rho^{-\nu}v)^{\frac{1}{2(1-\nu)}}(\rho^{-1}v)^{\frac{3-4\nu}{2(1-\nu)}},\end{align*}
and hence, by \eqref{eq-estimate-cylinderical-boundary-pre} and \eqref{eq-estimate-gradient-v1},
$$\|F(v)(t,\cdot)\|_{L^{\frac{4(1-\nu)}{3-4\nu}}(\Sigma)}
\le  C(e^{-\gamma t})^{\frac{1}{2(1-\nu)}} \|\rho^{-1}v(t,\cdot)\|^{\frac{3-4\nu}{2(1-\nu)}}_{L^2(\Sigma)}
\le Ce^{-2\gamma t}.$$
Note that $\frac{4(1-\nu)}{3-4\nu}>2$, since $\nu>1/2$ by Theorem \ref{thrm-regularity-boundary-Lip-pre}.
If $\nu\ge 3/4$, we get a similar estimate of 
the $L^\infty$-norm.
\end{proof} 

As we see in the proof, we can take $p=\infty$ for $n\ge 6$, $p=4$ for $n= 5$, 
and some universal $p>2$ for $n=3$. 
These choices of $p$ are universal, independent of 
specific domains $\Sigma$. 
The discussion of the case $n=3$ is quite subtle. It depends 
essentially on the estimate of the H\"older index that $\nu>1/2$.

\smallskip 

In the rest of the this section, we discuss optimal regularity near boundary. 
Recall that $\kappa$ and $\beta$ are given by \eqref{eq-definition-constants}. Then, 
the positive $s$ satisfying \eqref{eq-definition-s} is in fact given by 
\begin{equation}\label{eq-definition-s-explicit}
s=\frac12(n+2).\end{equation}
We now improve Lemma \ref{lemma-estimate-cylinderical-boundary}. 

\begin{lemma}\label{lemma-estimate-cylinderical-boundary-improved} 
Let $\Sigma\subsetneq\mathbb{S}^{n-1}$ be a 
Lipschitz domain, 
and $\rho\in C^\infty(\Sigma)\cap \mathrm{Lip}(\Sigma)$ be the positive solution of 
\eqref{eq-LN-domain-eq-rho-new}-\eqref{eq-LN-domain-rho-boundary-new}, 
satisfying \eqref{eq-condition-gradient}, 
for some constants $c_0\ge 0$ and $\alpha\in (0,1]$. 
Assume that $v$ is a smooth solution of \eqref{eq-basic-equation-v-new} in $(T,\infty)\times\Sigma$
satisfying \eqref{eq-estimate-v-basic-gamma}, 
for some constants $\gamma\ge\beta$ and $A>0$. 
Then, 
\begin{equation}\label{eq-estimate-cylinderical-boundary-pre-improved} 
|v|+\rho|\nabla_{(t,\theta)}v|
\le Ce^{-\gamma t}\rho^{s}
\quad\text{in }(T+1,\infty)\times\Sigma,\end{equation}
where $C$ is a positive constant  depending only on $n$, $\gamma$, $A$ and $\Sigma$.
Moreover, 
$v$ is Lipschitz in $(T,\infty)\times\Sigma$.
\end{lemma} 

\begin{proof} 
For some $a> 0$, assume 
\begin{equation}\label{eq-iteration-rho1-improved}|v|\le Ce^{-\gamma t}\rho^a.\end{equation}
Then, by \eqref{eq-def-mathcal-F-new}, we have 
\begin{equation}\label{eq-iteration-rho2-improved}
|F(v)|\le Ce^{-2\gamma t}\rho^{\beta+2a-2}
\le Ce^{-\gamma t}\rho^{a+\alpha-2},\end{equation}
where $\alpha$ above is the minimum of $\beta$ and $\alpha$ in \eqref{eq-condition-gradient}. 
Obviously, $\alpha\in (0,1]$. It will be apparent soon why we replace the power $\beta+2a$ by a smaller one 
$a+\alpha$. 
We now proceed to prove, for some $b>a$,  
\begin{equation}\label{eq-estimate-final-linear}|v|\le 
Ce^{-\gamma t}\rho^{b}\quad\text{in }(T+1,\infty)\times \Sigma.\end{equation}
We consider two cases: 

Case 1: If $0<a<a+\alpha<s$,  we can take $b=a+\alpha$; 

Case 2: If $a<s<a+\alpha$, we can take $b=s$. 

\noindent 
Suppose this is already done. 
We note that \eqref{eq-iteration-rho1-improved} holds for $a=\nu$ 
by  Lemma \ref{lemma-estimate-cylinderical-boundary}. 
We can start to iterate. After finitely many steps and by adjusting the initial $a$ if necessary, 
we obtain the estimate of $v$ in \eqref{eq-estimate-cylinderical-boundary-pre-improved}. 
Then, the estimate of the gradient $\nabla_{(t,\theta)} v$  in \eqref{eq-estimate-cylinderical-boundary-pre-improved}
follows from the interior $C^{1}$-estimate, as 
in the proof of Theorem \ref{thrm-regularity-boundary-Lip-pre}.

The proof of both Case 1 and Case 2 is similar as that 
of Lemma \ref{lemma-estimate-cylinderical-boundary}. 
Fix a $t_0>T+1$ and set, for some $\rho_0$,  
$$\Omega_0=\{(t, \theta); |t-t_0|<1, 0<\rho<\rho_0\}.$$

We first consider Case 1, $0<a<a+\alpha<s$ 
For some $b>a$ and 
some $\varepsilon>0$ 
to be determined, we set
\begin{equation}\label{eq-barrier-case2}
w=e^{-\gamma t}\big[\varepsilon (t-t_0)^2\rho^a+\rho^b\big].\end{equation}
If $a(a-1)<\kappa$ and $b(b-1)<\kappa$, 
a similar computation as in the proof of Lemma \ref{lemma-estimate-cylinderical-boundary} yields
\begin{align*} 
\mathcal L w
\le C_1\varepsilon e^{-\gamma t}\rho^{a+\alpha-2}+e^{-\gamma t}\rho^{b-2}\big[-[\kappa-b(b-1)]+C_2\rho^\alpha\big].
\end{align*}
We take $b=a+\alpha$. 
By taking $\varepsilon$ and $\rho_0$ 
small, we obtain 
\begin{align*} 
\mathcal L w\le 
-\frac12[\kappa-b(b-1)]e^{-\gamma t}\rho^{b-2}\quad\text{in }\Omega_0.
\end{align*}
By \eqref{eq-iteration-rho2-improved}, we get
$$|F|\le Ae^{-\gamma t}\rho^{b-2}.$$
By taking $M$ large, we obtain 
$$\mathcal L(Mw)\le -Ae^{-\gamma t}\rho^{b-2}\le \mathcal L(\pm v)\quad\text{in }\Omega_0.$$
Next, similarly as in \eqref{eq-boundary-comparison-pre0}, we have, for $M$ sufficiently large,  
\begin{align*}
\pm v\le Mw \quad\text{on }\partial\Omega_0.\end{align*}
By the maximum principle, we obtain 
$\pm v\le Mw$ in $\Omega_0$, and in particular at $t=t_0$, 
$$|v(t_0, \theta)|\le Me^{-\gamma t_0}\rho^{b}.$$
This holds for any $t_0>T+1$ and hence proves the estimate \eqref{eq-estimate-final-linear}
for any $t>T+1$.

Next, we consider Case 2, $a<s<a+\alpha$. Instead of \eqref{eq-barrier-case2}, we
consider, for some $\tau>s$, 
\begin{equation}\label{eq-barrier-case3}
w=e^{-\gamma t}\big[\varepsilon (t-t_0)^2\rho^a+\rho^s-\rho^\tau\big].\end{equation}
The proof of the estimate \eqref{eq-estimate-final-linear} for $b=s$ can be modified easily, 
with a combination of some calculations in the proofs of Case 1 and Theorem 
\ref{thrm-regularity-boundary-Lip}. 
\end{proof} 

In Lemma \ref{lemma-estimate-cylinderical-boundary-improved}, the best decay estimate near $\partial\Sigma$ 
is given by $\rho^{s}$. This is consistent with the decay estimate of eigenfunctions near $\partial\Sigma$ 
as in Theorem \ref{thrm-regularity-boundary-Lip}. 
Lemma \ref{lemma-estimate-cylinderical-boundary-improved} simply says that 
Lemma \ref{lemma-estimate-cylinderical-boundary} holds for $\nu=s$
under the additional assumption \eqref{eq-condition-gradient}. 

We now make some remarks. There are two main issues in this section, integration by parts 
on each slice $\{t\}\times\Sigma$ and the boundedness or the integrability of the nonlinear term $F(v)$, 
which are essentially needed in the proof of Theorem \ref{thrm-Asymptotic-order-gamma1-optimal}
and Lemma \ref{lemma-nonhomogeneous-linearized-eq-Lip} later on. 
If $\Sigma$ is merely Lipschitz, Lemma \ref{lemma-estimate-cylinderical-boundary} 
asserts that $v$ is globally H\"older continuous, with possibly a small H\"older index $\nu$. 
Such smallness does not allow us to integrate by parts. In Corollary \ref{cor-regularity-v-H10}, 
we proved that $v$ is $H_0^1$ when restricted to each slice $\{t\}\times\Sigma$, and hence 
we are able to perform integration by parts. Next, concerning the function $F(v)$, 
by \eqref{eq-def-mathcal-F-new} and  \eqref{eq-estimate-cylinderical-boundary-pre}, we have 
\begin{equation}\label{eq-estimate-basic-v}
|F(v)|\le C\rho^{\beta-2}v^2\le Ce^{-2\gamma t}\rho^{\beta-2+2\nu}.\end{equation}
If $n\ge 6$, then $\beta\ge 2$, and hence $F(v)$ is bounded. 
In Corollary \ref{cor-F(v)-integrability}, we proved that $F(v)$ is actually $L^p$-integrable, for some 
$p>\max\{2,n/2\}$ 
in the case $3\le n\le 5$. The assertion $p>n/2$ allows us to apply a well-known $L^\infty$-estimates. 

We next examine the case when  \eqref{eq-condition-gradient} is assumed. 
First, Lemma \ref{lemma-estimate-cylinderical-boundary-improved} asserts that $v$ is Lipschitz, 
and then allows us to perform integration by parts. 
Next, with $\nu=s$ as in Lemma \ref{lemma-estimate-cylinderical-boundary-improved}, 
we have $\beta+2\nu-2>0$, and hence $F(v)$ is always bounded by \eqref{eq-estimate-basic-v}. 
In conclusion, when \eqref{eq-condition-gradient} is assumed, 
Lemma \ref{lemma-regularity-v-derivatives-boundary} 
and Corollaries \ref{cor-regularity-v-H10}-\ref{cor-F(v)-integrability} are not needed.

\section{An Optimal Estimate in the $t$-direction}
\label{section-optimal-t}

In this section and the next, we continue to study the Yamabe equation in cylinders
as introduced in the previous section, 
and concentrate on expansions as $t\to\infty$.   
We adopt notations in the previous section. 

Let $v$ be a smooth solution of \eqref{eq-basic-equation-v-new} in $(T,\infty)\times\Sigma$
satisfying 
\begin{equation}\label{eq-estimate-v-fundamental}
|v|\le C_0 e^{-\beta t}\quad\text{in }(T,\infty)\times\Sigma,\end{equation}
for some positive constant $C_0$. In other words, 
\eqref{eq-estimate-v-basic-gamma} holds for $\gamma=\beta$. 
Hence, all results proved in the previous section hold for $\gamma=\beta$. 
Recall that $\beta=\frac{n-2}{2}$ by \eqref{eq-definition-constants}. 
The assumption \eqref{eq-estimate-v-fundamental}
demonstrates that $v$ decays exponentially 
at the rate $\beta$ as $t\to\infty$. Next, we improve the exponential decay in the $t$-direction
and derive an optimal decay rate. 
In the following result, $\lambda_1$ is the first eigenvalue  
as in Theorem \ref{thrm-singular-eigenvalue}. 

\begin{theorem}\label{thrm-Asymptotic-order-gamma1-optimal}
Let $\Sigma\subsetneq\mathbb{S}^{n-1}$ be a 
Lipschitz domain, 
and $\rho\in C^\infty(\Sigma)\cap \mathrm{Lip}(\Sigma)$ be the positive solution of 
\eqref{eq-LN-domain-eq-rho-new}-\eqref{eq-LN-domain-rho-boundary-new}. 
Assume that $v$ is a smooth solution of \eqref{eq-basic-equation-v-new} in $(T,\infty)\times\Sigma$
satisfying \eqref{eq-estimate-v-fundamental}. 
Then, 
\begin{equation}\label{eq-U2-0-optimal}
|v|\le Ce^{-\gamma_1 t}\rho^{\nu}
\quad\text{in } (T_*,\infty)\times\Sigma,\end{equation} 
where $\gamma_1=\sqrt{\lambda_1+\beta^2}$, 
$\nu$ is the positive constant as in Theorem \ref{thrm-regularity-boundary-Lip-pre}, and 
$T_*$ and $C$ are positive constants depending only on $n$, $C_0$, and $\Sigma$. 
\end{theorem}

\begin{proof} 
We adopt notations in the proof of Lemma \ref{lemma-regularity-v-derivatives-boundary} and 
Corollary \ref{cor-regularity-v-H10}, and divide the proof into several steps. 

{\it Step 1.} We write \eqref{eq-basic-equation-v-new} as \eqref{eq-equation-v-homo}, i.e., 
\begin{equation}\label{eq-equation-v-homo-1} 
\partial_{tt}v+\Delta_\theta v-\frac{1}{\rho^2}\big(\kappa+\epsilon(\rho^\beta v)\big)v-\beta^2v=0.\end{equation}
For simplicity, we write $v(t)=v(t,\cdot)$.
For any $t\in (T,\infty)$, we have $v(t)\in H_0^1(\Sigma)$ by Corollary \ref{cor-regularity-v-H10}. 
We write 
\eqref{eq-integral-identity} as 
\begin{align}\label{eq-integral-identity1}\int_{\Sigma}\big[ v_{tt}(t) v(t)
-\beta^2 v^2(t)\big]d\theta
= \int_{\Sigma}\big[|\nabla_\theta v(t)|^2+\big(\kappa+\epsilon(\rho^\beta v(t))\big)\rho^{-2} v^2(t)\big]d\theta.
\end{align}

{\it Step 2.} We claim that there exists a $T_*>T$ such that, 
for any $t>T_*$,
\begin{equation}\label{eq-estimate-L2-hat-v-Lip1-optimal}\|v(t,\cdot)\|_{L^2(\Sigma)}\le Ce^{-\gamma_1 t}.
\end{equation}
To prove \eqref{eq-estimate-L2-hat-v-Lip1-optimal}, set
$$y(t)=\Big[\int_{\Sigma}v^2(t)d\theta\Big]^{1/2}.$$
Then, 
$$y(t)y'(t)=\int_{\Sigma} v(t) v_t(t)d\theta,$$
and 
$$y(t)y''(t)+[y'(t)]^2=\int_{\Sigma}
\big[ v(t) v_{tt}(t)+ v_t^2(t)\big]d\theta.$$
The Cauchy inequality implies that, if  $y(t)>0$, then
$$[y'(t)]^2\le \int_{\Sigma} v_t^2(t)d\theta,$$
and hence
$$y(t)y''(t)\ge \int_{\Sigma} v(t) v_{tt}(t)d\theta.$$
By \eqref{eq-definition-c} and \eqref{eq-estimate-v-fundamental}, we have 
$$|\epsilon(\rho^\beta v(t))|\le C\rho^\beta |v(t)|\le \kappa C_* e^{-\beta t}.$$
We take $T_*$ so that $C_*e^{-\beta T_*}<1/2$. 
Since $v(t)\in H_0^1(\Sigma)$, we have, for any $t>T_*$, 
\begin{align*}
&\int_{\Sigma}\big[|\nabla_\theta v(t)|^2+\big(\kappa+\epsilon(\rho^\beta v(t))\big)\rho^{-2} v^2(t)\big]d\theta\\
&\qquad \ge (1-C_*e^{-\beta t})\int_{\Sigma}\big[|\nabla_\theta v(t)|^2+\kappa\rho^{-2} v^2(t)\big]d\theta\\
&\qquad\ge  \lambda_{1}(1-C_*e^{-\beta t})\int_{\Sigma}v^2(t)d\theta.
\end{align*}
By simple substitutions in \eqref{eq-integral-identity1}, we get 
$$y(t)y''(t)-\big(\lambda_{1}+\beta^2)y^2(t)+\lambda_{1}C_*e^{-\beta t}y^2(t)\ge 0.$$
With $\gamma_1=\sqrt{\lambda_{1}+\beta^2}$, set 
$$L_*w=w''-(\gamma^2_1-\lambda_{1}C_*e^{-\beta t})w.$$ 
If $y(t)>0$, we obtain 
$$L_{*}y\ge 0.$$
Note that $\lambda_{1}C_*e^{-\beta T_*}<\gamma^2_1/2$.

Set 
\begin{align}\label{eq-definition-comparison-z11-optimal}
z(t)=Ce^{-\gamma_1 t}\arctan t.\end{align}
A straightforward computation yields 
\begin{align*}
L_*z&=Ce^{-\gamma_1 t}\Big[-\frac{2t}{(1+t^2)^2}-\frac{2\gamma_1}{1+t^2}+\lambda_{1}C_*e^{-\beta t}\arctan t\Big]\\
&\le \frac{Ce^{-\gamma_1 t}}{1+t^2}\big[-2\gamma_1+\lambda_{1}C_*(1+t^2)e^{-\beta t}\arctan t\big]<0,
\end{align*}
if $t>T_*$ for  $T_*$ large further. 
Then, for some constant $C$ sufficiently large, we have $z(T_*)\ge y(T_*)$ and, if $y(t)>0$,  
\begin{equation}\label{eq-comparison-operators1-optimal}L_{*}(z-y)(t)\le 0.\end{equation}
Note that $y(t)\to 0$ as $t\to\infty$. If $z-y$ is negative somewhere on $(T_*,\infty)$, then the minimum of $z-y$ 
is negative and attained at some $t_*\in (T_*,\infty)$. Hence, $y(t_*)>z(t_*)>0$ 
and \eqref{eq-comparison-operators1-optimal}
holds at $t_*$. However, $(z-y)(t_*)<0$, $(z-y)'(t_*)=0$, 
and $(z-y)''(t_*)\ge0$, contradicting \eqref{eq-comparison-operators1-optimal}. Therefore, for any $t>T_*$,
$y(t)\le z(t)$, and hence \eqref{eq-estimate-L2-hat-v-Lip1-optimal} holds. 

{\it Step 3.} We now prove \eqref{eq-U2-0-optimal} for any $t>T_*+1$.
First, by applying the De Giorgi-Moser $L^\infty$-estimates to the equation \eqref{eq-equation-v-homo-1}, 
we obtain, for any $t>T_*+1$, 
\begin{equation}\label{eq-interior-estimate-cylinder-pre1-optimal}
\sup_{\{t\}\times\Sigma}|v|\le 
C\|v\|_{L^2((t-1,t+1)\times\Sigma)}
\le Ce^{-\gamma_1 t}.\end{equation}
To get estimate \eqref{eq-interior-estimate-cylinder-pre1-optimal}, 
we only need to introduce cutoff functions along the $t$-direction, 
since $v(t)\in H_0^1(\Sigma)$ for any $t>T_*$.  
Also note that the coefficient of the term $\rho^{-2}v$ in the equation \eqref{eq-equation-v-homo-1} has a good sign.
Hence, the corresponding term can be simply dropped in integration by parts.  
With \eqref{eq-interior-estimate-cylinder-pre1-optimal}, 
we have the desired results by applying Lemma \ref{lemma-estimate-cylinderical-boundary} 
with $\gamma=\gamma_1$. 
\end{proof}

In the proof of Theorem \ref{thrm-Asymptotic-order-gamma1-optimal}, we performed integration by parts 
on each slice $\{t\}\times\Sigma$, the justification of which is provided by the assertion 
that $v(t,\cdot)\in H_0^1(\Sigma)$ by 
Corollary \ref{cor-regularity-v-H10}. 
We also point out that the exponential decay rate $\gamma_1$ in $t$ in \eqref{eq-U2-0-optimal} is optimal. 

Under the additional assumption that  
$\rho$ also satisfies \eqref{eq-condition-gradient}, 
Theorem \ref{thrm-Asymptotic-order-gamma1-optimal} holds for $\nu=s$. 
In this case, \eqref{eq-U2-0-optimal} has the form
\begin{equation*}
\big|v|\le Ce^{-\gamma_1 t}\rho^{s}\quad\text{in }(T_*,\infty)\times\Sigma.\end{equation*} 
This is an optimal estimate both in $t$ and near $\partial\Sigma$.

\section{Asymptotic Expansions}\label{sec-iterations}

In this section, we continue to study the Yamabe equation in  cylinders.  
We first discuss the corresponding linear equations and then study 
the nonlinear equation \eqref{eq-basic-equation-v-new}.

Let $L$ be the operator given by \eqref{eq-def-L-Sigma}. 
By Theorem \ref{thrm-singular-eigenvalue}, 
there exist an increasing sequence of positive constants $\{\lambda_i\}_{i\ge 1}$, divergent to $\infty$, 
and an $L^2(\Sigma)$-orthonormal basis  
$\{\phi_i\}_{i\ge 1}$ 
such that, for $i\ge 1$,  $\phi_i\in C^\infty(\Sigma)\cap H_0^1(\Sigma)$ and 
$L \phi_i=-\lambda_i\phi_i.$
By Theorem \ref{thrm-regularity-boundary-Lip-pre}, $\phi_i\in C^\nu(\bar\Sigma)$ for some $\nu>0$. 
In the following, 
we fix such a sequence $\{\phi_i\}$. 

Let $\mathcal L$ be the operator given by \eqref{eq-def-mathcal-L-new1}. Then, 
\begin{equation}\label{eq-def-mathcal-L-new1a}
\mathcal Lv=\partial_{tt}v+Lv
-\beta^2 v.\end{equation}
For each $i\ge 1$ and any $\psi=\psi(t)\in C^2(\mathbb R)$, we write 
\begin{equation}\label{eq-U2-01a}\mathcal L(\psi \phi_i)=(L_i\psi)\phi_i,\end{equation} 
where $L_i$ is given by 
\begin{equation}\label{eq-U2-01b}
L_i\psi=\psi_{tt}-(\lambda_i+\beta^2)\psi.\end{equation}
Then, the kernel $\mathrm{Ker}(L_i)$ has a basis  
$e^{-\gamma_i t}$ and $e^{\gamma_i t}$, with 
\begin{equation}\label{eq-exponential-decay-order} 
\gamma_i=\sqrt{\lambda_i+\beta^2}.\end{equation} 
We note that $e^{-\gamma_i t}$ decays exponentially and $e^{\gamma_i t}$ grows exponentially as $t\to\infty$, 
for each $i\ge 1$,  
and that 
$\{\gamma_i\}_{i\ge 1}$ is increasing and divergent to $\infty$. 

For each fixed $i\ge 1$, we now analyze the linear equation 
\begin{equation}\label{eq-linearization-ODE-i}
L_i\psi=f\quad\text{on }(T,\infty).
\end{equation}
We study a class of solutions of \eqref{eq-linearization-ODE-i}, which converge to zero as $t\to\infty$. 

\begin{lemma}\label{lemma-decay-sol-Ljw=f-ODE}
Let 
$\gamma>0$ be a constant,  $m\geq 0$ be an integer, 
and $f$ be a smooth function on $(T,\infty)$, satisfying, 
for any $t>T$, 
$$|f(t)|\leq C t^me^{-\gamma t}.$$
Let $\psi$ be a smooth solution of \eqref{eq-linearization-ODE-i} on $(T,\infty)$  such that 
$\psi(t)\to 0$ as $t\to\infty$. 
Then, 
for any $t>T$,
\begin{equation*}
|\psi(t)|\leq\left\{
\begin{aligned}
&Ct^me^{-\gamma t}\hspace{1cm}\text{if } \gamma< \gamma_i,\\
&Ct^{m+1}e^{-\gamma t}\hspace{0.6cm}\text{if } \gamma=\gamma_i,
\end{aligned}
\right.
\end{equation*}
and there exists a constant $c$ such that
\begin{equation*}
|\psi(t)-ce^{-\gamma_it}|\leq
Ct^me^{-\gamma t}\hspace{0.6cm}\text{if } \gamma> \gamma_i.
\end{equation*}
\end{lemma}

\begin{proof} We first construct a particular solution with a desired decay rate as $t\to\infty$. 
We claim that 
there exists a smooth solution $\psi_{*}$ of 
\eqref{eq-linearization-ODE-i}  on $(T,\infty)$  such that,  
for any $t>T$,
$$|\psi_{*}(t)|\le \begin{cases} Ct^me^{-\gamma t} &\text{if }\gamma\neq\gamma_i,\\ 
Ct^{m+1}e^{-\gamma t}&\text{if }\gamma=\gamma_i.\end{cases}$$
The proof is standard. Since the kernel $\mathrm{Ker}(L_i)$ has a basis  
$e^{-\gamma_i t}$ and $e^{\gamma_i t}$, 
a particular solution $\psi_*$ of \eqref{eq-linearization-ODE-i}  can be given by the following: 
for $\gamma\le\gamma_i$, 
$$\psi_*(t)=-\frac{e^{-\gamma_i t}}{2\gamma_i}\int_{T}^te^{\gamma_is}f(s)ds-
\frac{e^{\gamma_it}}{2\gamma_i}\int_{t}^\infty e^{-\gamma_i s}f(s)ds,$$
and, for $\gamma>\gamma_i$, 
$$\psi_*(t)=\frac{e^{-\gamma_i t}}{2\gamma_i}\int_{t}^\infty e^{\gamma_is}f(s)ds-
\frac{e^{\gamma_it}}{2\gamma_i}\int_{t}^\infty e^{-\gamma_i s}f(s)ds.$$
This particular solution satisfies the desired estimate. 

Next, any solution $\psi$ of \eqref{eq-linearization-ODE-i} can be written as 
$$\psi=c_1e^{-\gamma_it}+c_2e^{\gamma_it}+\psi_*,$$
for some constants $c_1$ and $c_2$, and the function $\psi_*$ just constructed. 
Since $\psi(t)\to0$ as $t\to\infty$, then $c_2=0$. 
Therefore, we have  the desired estimate.
\end{proof}

Now, we study the linear equation
\begin{equation}\label{eq-linearization-PDE-spherical}\mathcal Lv=f\quad\text{in }
(T,\infty)\times\Sigma.\end{equation}
We discuss asymptotic expansions of solutions of \eqref{eq-linearization-PDE-spherical}
as $t\to \infty$. 

\begin{lemma}\label{lemma-nonhomogeneous-linearized-eq-Lip}
Let $\Sigma\subsetneq\mathbb{S}^{n-1}$ be a 
Lipschitz domain, 
$\rho\in C^\infty(\Sigma)\cap \mathrm{Lip}(\bar\Sigma)$ be the positive 
solution of \eqref{eq-LN-domain-eq-rho-new}-\eqref{eq-LN-domain-rho-boundary-new}, 
and  $p$ be a constant such that $p>\max\{2,n/2\}$. 
Let $\nu$ be the constant in Theorem \ref{thrm-regularity-boundary-Lip-pre}, 
and 
$f$ be a smooth function in $(T,\infty)\times\Sigma$, satisfying, 
for some positive constants $\gamma$, $A$, and some integer $m\geq 0$,  
\begin{align}\label{eq-assumption-bound-f} 
|f|\le At^me^{-\gamma t}\rho^{\nu-2}\quad\text{in } (T,\infty)\times\Sigma,\end{align}
and,  for any  $t\in (T,\infty)$, 
\begin{align}\label{eq-assumption-bound-f-integral} 
\|f(t,\cdot)\|_{L^{2}(\Sigma)}
+\|f(t,\cdot)\|_{L^{p}(\Sigma)}\le At^me^{-\gamma t}.\end{align}
Let $v\in C^\infty((T,\infty)\times\Sigma)\cap C^\nu((T,\infty)\times\bar\Sigma)$ 
be a solution of \eqref{eq-linearization-PDE-spherical} 
in $(T,\infty)\times\Sigma$ such that 
$v=0$ on $(T,\infty)\times\partial\Sigma$, 
$v(t,\theta)\to 0$ as $t\to\infty$ uniformly for $\theta\in\Sigma$, 
and $v(t,\cdot)\in H_0^1(\Sigma)$ and $(Lv)(t,\cdot)\in L^2(\Sigma)$, for any $t>T$. 

$\mathrm{(i)}$ If $\gamma\le \gamma_1$, then, 
for any $(t,\theta)\in (T+1,\infty)\times\Sigma$, 
$$|v|\le \begin{cases} Ct^me^{-\gamma t}\rho^\nu &\text{if }\gamma<\gamma_1,\\ 
Ct^{m+1}e^{-\gamma t}\rho^\nu&\text{if }\gamma=\gamma_1.\end{cases}$$

$\mathrm{(ii)}$ If $\gamma_l<\gamma\le \gamma_{l+1}$ for some positive integer $l$, 
then, 
for any $(t,\theta)\in (T+1,\infty)\times\Sigma$, 
$$\Big|v-\sum_{i=1}^lc_ie^{-\gamma_i t}\phi_i\Big|\le 
\begin{cases} Ct^me^{-\gamma t}\rho^\nu &\text{if }\rho_l<\gamma<\rho_{l+1},\\ 
Ct^{m+1}e^{-\gamma t}\rho^\nu&\text{if }\gamma=\rho_{l+1},\end{cases}$$
for some constant $c_i$, for 
$i=1, \cdots, l$. 
\end{lemma}

\begin{proof} The proof 
is similar as that of Theorem \ref{thrm-Asymptotic-order-gamma1-optimal}. 
We consider (ii) only,  $\gamma_l<\gamma\le \gamma_{l+1}$ for some positive integer $l$. 
The proof consists of several steps.

{\it Step 1.} For each $i= 1, \cdots, l$, and $t>T$,
set 
$$v_{i}(t)=\int_{\Sigma}v(t, \theta)\phi_{i}(\theta)d\theta,\quad 
f_{i}(t)=\int_{\Sigma}f(t, \theta)\phi_{i}(\theta)d\theta.$$
Then, by \eqref{eq-assumption-bound-f-integral}, 
\begin{equation}\label{eq-condition-f-i}|f_{i}(t)|\le A_it^me^{-\gamma t}.\end{equation}
By multiplying \eqref{eq-linearization-PDE-spherical} by $\phi_{i}$ and integrating over $\{t\}\times\Sigma$
for any $t>T$, we obtain 
$$L_iv_{i}= f_{i}\quad\text{on }(T,\infty).$$
We point out that the integration is performed on $\{t\}\times\Sigma$ for each $t>T$ and 
that the boundary integrals vanish by $v(t,\cdot), \phi_i\in H_0^1(\Sigma)$. 
Since $v(t,\theta)\to 0$ as $t\to\infty$ uniformly for $\theta\in\Sigma$, 
then $v_{i}(t)\to 0$ as $t\to\infty$. 
Note that $\gamma> \gamma_i$ for $i=1, \cdots, l$. 
By applying Lemma \ref{lemma-decay-sol-Ljw=f-ODE} to $L_i$, we conclude that
there exists a constant $c_{i}$ for $i=1, \cdots, l$ such that, 
for any $t>T$,
\begin{equation}\label{eq-U2-03-0b-Lip}
|v_{i}(t)-c_{i}e^{-\gamma_i t}|\leq C_it^me^{-\gamma t}.
\end{equation}
Set 
\begin{align}\label{eq-def-hat-w-Lip}\begin{split}\widehat v(t,\theta)&=v(t,\theta)
-\sum_{i=1}^{l}v_{i}(t)\phi_{i}(\theta),\\
\widehat f(t,\theta)&=f(t,\theta)
-\sum_{i=1}^{l}f_{i}(t)\phi_{i}(\theta).\end{split}\end{align}
A simple calculation yields 
\begin{equation}\label{eq-U2-03-1-Lip}\mathcal L\widehat v=\widehat f\quad\text{in }(T,\infty)\times\Sigma.
\end{equation}
For simplicity, we write $\widehat v(t)=\widehat v(t,\cdot)$ and $\widehat f(t)=\widehat f(t,\cdot)$.

{\it Step 2.} Set $m_*=m$ if $\gamma_l<\gamma<\gamma_{l+1}$ 
and $m_*=m+1$ if $\gamma=\gamma_{l+1}$. 
We proceed to prove that there exists a $T_*>T$ such that, for any $t>T_*$,  
\begin{equation}\label{eq-estimate-L2-hat-v-Lip}\|\widehat v(t,\cdot)\|_{L^2(\Sigma)}\le Ct^{m_*}e^{-\gamma t}.
\end{equation}
First, we note that $(L\widehat v)(t)\in L^2(\Sigma)$ and $\widehat v(t)\in H_0^1(\Sigma)$ for each $t>T$. 
By multiplying \eqref{eq-U2-03-1-Lip} by $\widehat v(t)$ and integrating over $\Sigma$, we obtain 
\begin{align*}&\int_{\Sigma}\big[\widehat v_{tt}(t)\widehat v(t)
-\beta^2\widehat v^2(t)\big]d\theta
- \int_{\Sigma}\big[|\nabla_\theta\widehat v(t)|^2+\kappa\rho^{-2}\widehat v^2(t)\big]d\theta
=\int_{\Sigma}\widehat f(t)\widehat v(t)d\theta.
\end{align*}
Set
$$y(t)=\Big[\int_{\Sigma}\widehat v^2(t)d\theta\Big]^{1/2}.$$
Then, 
$$y(t)y'(t)=\int_{\Sigma} \widehat v(t) \widehat v_t(t)d\theta,$$
and 
$$y(t)y''(t)+[y'(t)]^2=\int_{\Sigma}
\big[ \widehat v(t) \widehat v_{tt}(t)+ \widehat v_t^2(t)\big]d\theta.$$
The Cauchy inequality implies that, if  $y(t)>0$, then
$$[y'(t)]^2\le \int_{\Sigma} \widehat v_t^2(t)d\theta,$$
and hence
$$y(t)y''(t)\ge \int_{\Sigma} \widehat v(t) \widehat v_{tt}(t)d\theta.$$
Since $\widehat v(t)\perp \phi_i$ in $L^2(\Sigma)$ for each $i=1, \cdots, l$ and any $t>T$, then 
$$\int_{\Sigma}\big[|\nabla_\theta\widehat v(t)|^2+\kappa\rho^{-2}\widehat v^2(t)\big]d\theta\ge \lambda_{l+1}
\int_{\Sigma}\widehat v^2(t)d\theta.$$
Moreover, by \eqref{eq-assumption-bound-f-integral}, 
$$\Big|\int_\Sigma \widehat f(t)\widehat v(t)d\theta\Big|
=\Big|\int_\Sigma  f(t) \widehat v(t)d\theta\Big|
\le C_0t^me^{-\gamma t}\|\widehat v(t)\|_{L^2(\Sigma)}.$$
A simple substitution yields 
$$y(t)y''(t)-\big(\lambda_{l+1}+\beta^2\big)y^2(t)\ge -C_0t^me^{-\gamma t}y(t).$$
If $y(t)>0$, we obtain 
$$L_{l+1}y\ge -C_0t^me^{-\gamma t}.$$

Set 
\begin{align}\label{eq-definition-comparison-z11}
z(t)=Ct^{m_*}e^{-\gamma t}.\end{align}
If $\gamma_l<\gamma<\gamma_{l+1}$,  then 
\begin{align*}
L_{l+1}(t^me^{-\gamma t})
=t^me^{-\gamma t}\big[\gamma^2-\gamma^2_{l+1}-mt^{-1}(2\gamma 
-(m-1)t^{-1})\big] 
\le -\frac12(\gamma^2_{l+1}-\gamma^2)t^me^{-\gamma t},
\end{align*}
if $t>T_*$ for some $T_*$ large. If $\gamma=\gamma_{l+1}$, then 
\begin{align*}
L_{l+1}(t^{m+1}e^{-\gamma t})=t^me^{-\gamma t}\big[-(m+1)(2\gamma
-mt^{-1})\big]\le -(m+1)\gamma t^me^{-\gamma t},
\end{align*}
if $t>T_*$ for some $T_*$ large. 
Hence, for some constant $C$ sufficiently large, we have $z(T_*)\ge y(T_*)$ and, if $y(t)>0$,  
\begin{equation}\label{eq-U2-03-2-Lip}L_{l+1}(z-y)(t)\le 0.\end{equation}
Note that $y(t)\to 0$ as $t\to\infty$. If $z-y$ is negative somewhere on $(T_*,\infty)$, then the minimum of $z-y$ 
is negative and attained at some $t_*\in (T_*,\infty)$. Hence, $y(t_*)>z(t_*)>0$ and \eqref{eq-U2-03-2-Lip}
holds at $t_*$. However, $(z-y)(t_*)<0$, $(z-y)'(t_*)=0$, 
and $(z-y)''(t_*)\ge0$, contradicting \eqref{eq-U2-03-2-Lip}. Therefore, for any $t>T_*$,
$y(t)\le z(t)$, and hence \eqref{eq-estimate-L2-hat-v-Lip} holds. 

{\it Step 3.} We now claim that, for any $(t,\theta)\in (T_*+2, \infty)\times\Sigma$, 
\begin{equation}\label{eq-U2-03-0d-Lip}|\widehat v(t,\theta)|\le Ct^{m_*}e^{-\gamma t}\rho^\nu.\end{equation}
By the definition of $\widehat f$, we have, for any $(t,\theta)\in (T_*, \infty)\times\Sigma$, 
\begin{equation}\label{eq-estimate-Lsup-hat-f-Lip}
|\widehat f(t,\theta)|\le Ct^me^{-\gamma t}\rho^{\nu-2},\end{equation}
and 
\begin{equation}\label{eq-estimate-L2-hat-f-Lip}
\|\widehat f(t,\cdot)\|_{L^p(\Sigma)}\le Ct^me^{-\gamma t}.\end{equation}
By $p>n/2$ and applying the $L^\infty$-estimates to the equation \eqref{eq-U2-03-1-Lip}, 
we obtain, for any $t>T_*+1$, 
\begin{equation}\label{eq-interior-estimate-cylinder-pre1-optimal-hat}
\sup_{\{t\}\times\Sigma}|\widehat v|\le 
C\big\{\|\widehat v\|_{L^2((t-1,t+1)\times\Sigma)}+\|\widehat f\|_{L^p((t-1,t+1)\times\Sigma)}\big\}
\le Ct^{m_*}e^{-\gamma t}.\end{equation}
With \eqref{eq-interior-estimate-cylinder-pre1-optimal-hat}, 
we have \eqref{eq-U2-03-0d-Lip} by 
proceeding similarly as in the proof of Lemma \ref{lemma-estimate-cylinderical-boundary}.


{\it Step 4.} We now finish the proof. We write
\begin{align}\label{eq-expansion-order-l-Lip}
v(t,\theta)-\sum_{i=1}^lc_{i}e^{-\gamma_it}\phi_{i}(\theta)
=\sum_{i=1}^l\big(v_{i}(t)-c_{i}e^{-\gamma_it}\big)\phi_{i}(\theta)  
+\widehat v(t,\theta).\end{align}
By combining  
\eqref{eq-U2-03-0b-Lip} and  
\eqref{eq-U2-03-0d-Lip}, 
we have the desired result. 
\end{proof}

We point out that the stated assumption $p>n/2$ is essentially needed in order to apply the 
$L^\infty$-estimates 
to get \eqref{eq-U2-03-0d-Lip} from the $L^2$-estimate   \eqref{eq-estimate-L2-hat-v-Lip}. 
Later on, we will apply Lemma \ref{lemma-nonhomogeneous-linearized-eq-Lip} to functions $f=F(v)$ 
as in Corollary \ref{cor-F(v)-integrability}.

\smallskip 

Next, we return to the nonlinear equation \eqref{eq-basic-equation-v-new} 
and discuss higher order expansions of its solutions along the $t$-direction. 
We first describe our strategy. 
Let $v$ be a smooth solution of \eqref{eq-basic-equation-v-new} in $(T,\infty)\times\Sigma$
satisfying \eqref{eq-estimate-v-fundamental}. 
We will apply Lemma 
\ref{lemma-nonhomogeneous-linearized-eq-Lip} to the equation $\mathcal Lv=F(v)$. 
Since $F(v)$ is nonlinear in $v$, we need to apply Lemma 
\ref{lemma-nonhomogeneous-linearized-eq-Lip} successively. In each step, we aim to 
get a decay estimate of $F(v)$, with a decay rate better than that of $v$. 
Then, we can subtract expressions with the lower decay rates generated by Lemma 
\ref{lemma-nonhomogeneous-linearized-eq-Lip} to improve the decay rate of $v$.
To carry out this process, we make two preparations. 

As the first preparation, we introduce the {\it index set} $\mathcal I$ by defining
\begin{align}\label{eq-def-index-Lip}
\mathcal I=\Big\{\sum_{i\ge 1} m_i\gamma_i;\, m_i\in \mathbb Z_+\text{ with finitely many }m_i>0\Big\}.
\end{align}
In other words, $\mathcal I$ is the collection of linear combinations of finitely many $\gamma_1, \gamma_2, \cdots$ 
with positive integer coefficients. 
It is possible that some $\gamma_i$ can be written 
as a linear combination of some of $\gamma_1, \cdots, \gamma_{i-1}$ with positive integer coefficients, 
whose sum is at least two. 

For the second preparation, we need to construct particular solutions. 

\begin{lemma}\label{lemma-particular-solutions} 
Let $\nu$ be as in Theorem \ref{thrm-regularity-boundary-Lip-pre}, 
$m$ be a nonnegative integer, $\gamma$ be a positive constant, 
and $h_0, h_1, \cdots, h_m$ $\in L^2(\Sigma)$. 
Then, there exist $w_0, w_1, \cdots, w_m, w_{m+1}\in H_0^1(\Sigma)$ 
such that 
\begin{equation}\label{eq-particular-solutions} 
\mathcal L\Big(\sum_{j=0}^{m+1} t^je^{-\gamma t}w_j\Big)=\sum_{j=0}^{m} t^je^{-\gamma t}h_j
\quad\text{in }\mathbb R\times\Sigma.\end{equation}
Moreover, if $|h_j|\le A\rho^{\nu-2}$ for any $j=0, 1, \cdots, m$ and some constant $A>0$, 
then $w_j\in C^\nu(\bar\Sigma)$ with $w_j=0$ on $\partial\Sigma$ for $j=0, 1, \cdots, m+1$. 
\end{lemma}

\begin{proof} For any nonnegative integer $j$, a straightforward calculation yields 
\begin{align*}
\mathcal L(t^je^{-\gamma t}w_j)=t^je^{-\gamma t}\big(Lw_j+(\gamma^2-\beta^2)w_j\big)
-2\gamma jt^{j-1}e^{-\gamma t}w_j+j(j-1)t^{j-2}e^{-\gamma t}w_j.
\end{align*}
We consider two cases. 

{\it Case 1: $\gamma^2-\beta^2\neq \lambda_l$ for any $l$.} 
For $w_0, w_1, \cdots, w_m$ to be determined, we consider 
\begin{equation*}
\mathcal L\Big(\sum_{j=0}^{m} t^je^{-\gamma t}w_j\Big)=\sum_{j=0}^{m} t^je^{-\gamma t}h_j
\quad\text{in }\mathbb R\times\Sigma.\end{equation*}
Note that 
\begin{align*}
\mathcal L\Big(\sum_{j=0}^{m} t^je^{-\gamma t}w_j\Big)
&=\sum_{j=0}^{m} t^je^{-\gamma t}\big(Lw_j+(\gamma^2-\beta^2)w_j\big)\\
&\qquad -\sum_{j=0}^{m-1}2(j+1)\gamma t^{j}e^{-\gamma t}w_{j+1}
+\sum_{j=0}^{m-2}(j+2)(j+1)t^{j}e^{-\gamma t}w_{j+2}.
\end{align*}
Hence, we take $w_{m+2}=w_{m+1}=0$ and solve for $w_m, \cdots, w_1, w_0$ 
inductively, such that, for $j=m, \cdots, 1, 0$, 
\begin{align*}
Lw_j+(\gamma^2-\beta^2)w_j
-2(j+1)\gamma w_{j+1}
+(j+2)(j+1)w_{j+2}=h_j.
\end{align*}
Theorem \ref{thrm-singular-Fredholm}(i) 
implies the existence of a unique $w_j\in H_0^1(\Sigma)$. 

{\it Case 2: $\gamma^2-\beta^2= \lambda_l$ for some $l$.} Let $\mathcal E$ be the eigenspace 
corresponding to $\lambda_l$. For each $j=0, 1, \cdots, m$, write 
$$h_j=\widehat h_j+\widetilde h_j,$$
with $\widehat h_j\in \mathcal E$ and $\widetilde h_j\in \mathcal E^{\perp}$. 
First, for $\widetilde w_0, \widetilde w_1, \cdots, \widetilde w_m$ to be determined, we consider 
\begin{equation*}
\mathcal L\Big(\sum_{j=0}^{m} t^je^{-\gamma t}\widetilde w_j\Big)=\sum_{j=0}^{m} t^je^{-\gamma t}\widetilde h_j
\quad\text{in }\mathbb R\times\Sigma.\end{equation*}
Since $\widetilde h_j\in \mathcal E^{\perp}$, we can find such $\widetilde w_j\in\mathcal E^{\perp}$
by Theorem \ref{thrm-singular-Fredholm}(ii), 
similarly as in Case 1.  
Next, for $\widehat w_1, \cdots, \widehat w_{m+1}$ to be determined, we consider 
\begin{equation*}
\mathcal L\Big(\sum_{j=1}^{m+1} t^je^{-\gamma t}\widehat w_j\Big)=\sum_{j=0}^{m} t^je^{-\gamma t}\widehat h_j
\quad\text{in }\mathbb R\times\Sigma.\end{equation*}
To this end, we take $\widehat w_{m+2}=0$ and solve for $\widehat w_{m+1}, \cdots, \widehat w_1$ 
inductively, such that, for $j=m, \cdots, 0$, 
\begin{align*}
-2(j+1)\gamma\widehat w_{j+1}
+(j+2)(j+1)\widehat w_{j+2}=\widehat h_j.
\end{align*}
Then, for each $j=1, \cdots, m+1$, $\widehat w_j\in \mathcal E$, and hence
$$L\widehat w_j+(\gamma^2-\beta^2)\widehat w_j=0.$$
In summary, we have 
\begin{align*}\mathcal L\Big(\sum_{j=0}^{m} t^je^{-\gamma t}\widetilde w_j
+\sum_{j=1}^{m+1} t^je^{-\gamma t}\widehat w_j\Big)
=\sum_{j=0}^{m} t^je^{-\gamma t}h_j.\end{align*}

By combining the two cases, we have the existence of the desired $w_j\in H_0^1(\Sigma)$. 
Theorem \ref{thrm-regularity-boundary-Lip-pre}  implies $w_j\in C^\nu(\bar\Sigma)$
and $w_j=0$ on $\partial\Sigma$, for $j=0, 1, \cdots, m+1$. \end{proof} 

Here, particular solutions were constructed as a simple application of the Fredholm alternative. 
In \cite{jiang}, similar particular solutions were constructed as infinite series 
and the convergence of such series is 
based on a growth estimate of eigenvalues. 



The main result in this section is the following expansions up to arbitrary orders.

\begin{theorem}\label{thrm-main-v-Lip} 
Let $\Sigma\subsetneq\mathbb{S}^{n-1}$ be a 
Lipschitz domain, 
and $\rho\in C^\infty(\Sigma)\cap \mathrm{Lip}(\Sigma)$ be the positive solution of 
\eqref{eq-LN-domain-eq-rho-new}-\eqref{eq-LN-domain-rho-boundary-new}. 
Assume that $v$ is a smooth solution of \eqref{eq-basic-equation-v-new} in $(T,\infty)\times\Sigma$
satisfying \eqref{eq-estimate-v-fundamental}. 
Then, there exist functions $c_{ij}\in C^\infty(\Sigma)\cap C^\nu(\bar\Sigma)$ with 
$c_{ij}=0$ on $\partial\Sigma$ such that, for any $m\ge 1$,  
\begin{equation}\label{eq-U2-0-z-Lip}
\Big|v-\sum_{i=1}^{m}\sum_{j=0}^{i-1}c_{ij}t^je^{-\mu_it}\Big|\le Ct^me^{-\mu_{m+1}t}\rho^{\nu}
\quad\text{in } (T+1,\infty)\times\Sigma,\end{equation} 
where $\nu$ is the positive constant as in Theorem \ref{thrm-regularity-boundary-Lip-pre}, and 
$C$ is a positive constant depending only on $n$, $m$, $C_0$, and $\Sigma$. 
\end{theorem}

We point out that there are two indices in the summation, $i$ for the decay rate 
in $e^{-\mu_i t}$ and $j$ for the polynomial growth 
rate in $t^j$. 

\begin{proof} 
Throughout the proof, we adopt the following notation: 
$f=O(h)$ if $|f|\le Ch$, for some positive constant $C$. 
All estimates in the following hold for 
any $t>T+1$ and any $\theta\in \mathbb S^{n-1}$. 
Take $p$ to be the constant as in Corollary \ref{cor-F(v)-integrability}.

Let $\{\phi_i\}$ be an orthonormal basis of $L^2(\Sigma)$, 
formed by eigenfunctions of $-L$, and $\{\lambda_i\}$ be the sequence of corresponding 
eigenvalues, arranged in an increasing order. 
Let $\mathcal L$ be the linear operator given by 
\eqref{eq-def-mathcal-L-new1a}, and $L_i$ be the projection of $\mathcal L$ given by \eqref{eq-U2-01b}. 
For each $i\ge 1$, 
there is an exponentially decaying solution $e^{-\gamma_i t}$ in $\mathrm{Ker}(L_i)$. 
Write 
$$F(v)=\rho^{-2-\beta}\sum_{i=2}^\infty b_i(\rho^\beta v)^i,$$
where $b_i$ is a constant, for each $i\ge 2$. We point out that we write 
the infinite sum just for convenience. We do not need the convergence of 
the infinite series and we always expand up to finite orders. 

Let $\mathcal I$ be the index set defined in \eqref{eq-def-index-Lip}. 
We decompose $\mathcal I$ by setting
$$\mathcal I_\gamma=\{\gamma_j:\, j\ge 1\},$$ 
and 
$$\mathcal I_{\widetilde \gamma}=\Big\{\sum_{i=1}^kn_i\gamma_i:\, n_i\in \mathbb Z_+, \sum_{i=1}^kn_i\ge 2\Big\}.$$
We assume $\mathcal I_{\widetilde \gamma}$ is given by a strictly increasing sequence 
$\{\widetilde\gamma_i\}_{i\ge 1}$, with 
$\widetilde \gamma_1=2\gamma_1$.

We first consider the case that 
\begin{equation}\label{eq-special}\mathcal I_\gamma\cap \mathcal I_{\widetilde \gamma}=\emptyset.\end{equation}
In other words, no $\gamma_i$ can be written 
as a linear combination of some of $\gamma_1, \cdots, \gamma_{i-1}$ with positive integer coefficients, except a single 
$\gamma_{i'}$ which is equal to $\gamma_i$. 
In this case, we arrange $\mathcal I$ as follows: 
\begin{equation}\label{eq-arrangement}
(\beta<)\,\gamma_1\le\cdots\le\gamma_{k_1}<\widetilde \gamma_1<\cdots<\widetilde\gamma_{l_1}<
\gamma_{k_1+1}\le\cdots\le\gamma_{k_2}<\widetilde \gamma_{l_1+1}<\cdots.\end{equation}
For each $\widetilde \gamma_i$, by the definition of $\mathcal I_{\widetilde \gamma}$, 
there are finitely many collections of 
nonnegative integers $n_1$, $\cdots$, $n_{k_1}$ satisfying
\begin{equation}\label{eq-requirement-m}
n_1+\cdots+n_{k_1}\ge 2, \quad n_1\gamma_1+\cdots+n_{k_1}\gamma_{k_1}= \widetilde \gamma_i.\end{equation}

By Theorem \ref{thrm-Asymptotic-order-gamma1-optimal}, 
we have 
\begin{equation}\label{eq-Asymptotic-Uk-02}v=O(e^{-\gamma_1 t}\rho^\nu).\end{equation}
We divide the proof for the case \eqref{eq-special} into several steps.

{\it Step 1.} Note $\gamma_{k_1}<\widetilde \gamma_1=2\gamma_1$. We claim that there exists 
a function $\eta_1$ such that 
$$v=\eta_1+O(e^{-\widetilde \gamma_1 t}\rho^{\nu}).$$
To prove this, we note that, by \eqref{eq-Asymptotic-Uk-02}, the explicit expression of 
$F$, and Corollary 
\ref{cor-F(v)-integrability}, 
\begin{equation}\label{eq-Asymptotic-Uk-03}\mathcal Lv=O(e^{-2\gamma_1t}\rho^{\beta+2\nu-2})
=O(e^{-\widetilde \gamma_1 t}\rho^{\beta+2\nu-2}),\end{equation}
and 
\begin{equation}\label{eq-estimate-L2-F}
\|\mathcal Lv(t,\cdot)\|_{L^p(\Sigma)}=O(e^{-\widetilde \gamma_1 t}).
\end{equation} 
We point out that we need to verify \eqref{eq-estimate-L2-F} only for $3\le n\le 5$. For $n\ge 6$,  $\beta\ge 2$   
and hence \eqref{eq-estimate-L2-F} is implied by \eqref{eq-Asymptotic-Uk-03}. 
By \eqref{eq-Asymptotic-Uk-03}, \eqref{eq-estimate-L2-F}, and 
Lemma \ref{lemma-nonhomogeneous-linearized-eq-Lip}(ii), we can take 
\begin{equation}\label{eq-Asymptotic-Uk-11}\eta_1(t,\theta)=\sum_{i=1}^{k_1}c_ie^{-\gamma_it}\phi_i(\theta),\end{equation}
for appropriate constant $c_i$. 
Set 
\begin{equation}\label{eq-Asymptotic-Uk-12}
v_1=v-\eta_1.\end{equation}
Then, $\mathcal L\eta_1=0$, $\mathcal Lv_1= F(v)$,  and 
\begin{equation}\label{eq-Asymptotic-Uk-13}v_1= O(e^{-\widetilde \gamma_1 t}\rho^{\nu}).\end{equation}
Note that \eqref{eq-Asymptotic-Uk-13} improves \eqref{eq-Asymptotic-Uk-02}. 

{\it Step 2.} We claim there exists an $\widetilde \eta_1$ such that, with  
\begin{equation}\label{eq-Asymptotic-Uk-21}
\widetilde v_1=v_1-\widetilde \eta_1=v-\eta_1-\widetilde \eta_1,\end{equation}
we have 
\begin{equation}\label{eq-Asymptotic-Uk-23}
\mathcal L\widetilde v_1= O(e^{-\widetilde \gamma_{l_1+1}t}\rho^{\beta+2\nu-2}),\end{equation}
and 
\begin{equation}\label{eq-Asymptotic-Uk-23-L2}
\mathcal \|(\mathcal L\widetilde v_1)(t, \cdot)\|_{L^p(\Sigma)}= O(e^{-\widetilde \gamma_{l_1+1}t}).\end{equation}
We will prove that $\widetilde \eta_1$ has the form 
\begin{equation}\label{eq-Asymptotic-Uk-22}
\widetilde\eta_{1}(t,\theta)=\sum_{i=1}^{l_1}
e^{-\widetilde \gamma_it}w_i(\theta),
\end{equation}
where $w_i\in C^\infty(\Sigma)\cap C^\nu(\bar\Sigma)\cap H_0^1(\Sigma)$ with $w_i=0$ on $\partial\Sigma$. 
Note that \eqref{eq-Asymptotic-Uk-23} improves \eqref{eq-Asymptotic-Uk-03}. 

To prove this, we take 
some function $\widetilde\eta_1$ 
to be determined, 
and then set $\widetilde v_1$ by \eqref{eq-Asymptotic-Uk-21}. 
Then, 
\begin{equation}\label{eq-Asymptotic-Uk-21a}
\mathcal L\widetilde v_1=F(v)-\mathcal L\widetilde\eta_1.\end{equation}
Note $3\gamma_1\in \mathcal I_{\widetilde \gamma}$. We discuss this step in several cases.

{\it Case 1. We assume $\gamma_{k_1+1}<3\gamma_1$.} Then, $\widetilde \gamma_{l_1}<\gamma_{k_1+1}<3\gamma_1$
and $\widetilde \gamma_{l_1+1}\le 3\gamma_1$. We now analyze $F(v)$ in \eqref{eq-Asymptotic-Uk-21a}. 
Note 
$$F(v)=F(v_1+\eta_1)=\rho^{-2-\beta}\sum_{i=2}^\infty b_i\rho^{i\beta}(v_1+\eta_1)^i.$$
It is worth mentioning again that we write the infinite sum just for convenience
and we always expand up to finite orders. 
For terms involving $v_1$, we have, by \eqref{eq-Asymptotic-Uk-13}, 
$$v_1^2\le Ce^{-4\gamma_1t}\rho^{2\nu},\quad |v_1\eta_1|\le Ce^{-3\gamma_1t}\rho^{2\nu}.$$
Note that $\eta_1$ is given by \eqref{eq-Asymptotic-Uk-11}. 
We write 
$$\sum_{i=2}^\infty b_i\rho^{i\beta}\eta_1^i=\sum_{n_1+\cdots+n_{k_1}\ge 2}a_{n_1\cdots n_{k_1}}
e^{-(n_1\gamma_1+\cdots+n_{k_1}\gamma_{k_1})t}\phi_1^{n_1}\cdots \phi_{k_1}^{n_{k_1}}
\rho^{(n_1+\cdots+n_{k_1})\beta},$$
where $n_1, \cdots, n_{k_1}$ are nonnegative integers, and $a_{n_1\cdots n_{k_1}}$ 
is a constant. By the definition of $\mathcal I_{\widetilde \gamma}$, 
$n_1\gamma_1+\cdots+n_{k_1}\gamma_{k_1}$ is some 
$\widetilde \gamma_i$. 
Hence, we can write 
\begin{equation}\label{eq-property-I}
\rho^{-2-\beta}\sum_{i=2}^\infty b_i\rho^{i\beta}\eta_1^i=\sum_{i=1}^{\infty}
e^{-\widetilde \gamma_it}h_i,\end{equation}
where $h_i$ is given by 
$$h_i=\rho^{-2-\beta}\sum_{(n_1, \cdots, n_{k_1})\in\mathcal N_{\widetilde \gamma_i}}
a_{n_1\cdots n_{k_1}}\phi_1^{n_1}\cdots \phi_{k_1}^{n_{k_1}}
\rho^{(n_1+\cdots+n_{k_1})\beta}.$$ 
Here, we denote by $\mathcal N_{\widetilde \gamma_i}$ the collection of all $(n_1, \cdots, n_{k_1})$ 
satisfying \eqref{eq-requirement-m}.
Then, 
$$|h_i|\le C\rho^{\beta+2\nu-2},$$
and $h_i$ has the same integrability as $F(v)$ in Corollary \ref{cor-F(v)-integrability}, i.e.,  
$h_i\in L^p(\Sigma)$ for some $p>n/2$. 
We now take the finite sum up to $l_1$ in the right-hand side of \eqref{eq-property-I} and denote it by 
$I_1$, i.e., 
\begin{equation}\label{eq-definitioni-I}I_1=
\sum_{i=1}^{l_1}
e^{-\widetilde \gamma_it}h_i.\end{equation}
Then, 
$$F(v)=I_1+O(e^{-\widetilde \gamma_{l_1+1}t}\rho^{\beta+2\nu-2}),$$
and hence, by \eqref{eq-Asymptotic-Uk-21a}, 
$$\mathcal L\widetilde v_1=I_1-\mathcal L\widetilde\eta_1+O(e^{-\widetilde \gamma_{l_1+1}t}\rho^{\beta+2\nu-2}).$$
Similar estimates for $L^p$-norms also hold. We now solve 
\begin{equation}\label{eq-solving-linear1}\mathcal L\widetilde\eta_1=I_1.\end{equation}
Note that $\gamma_m\neq\widetilde\gamma_i$ for any $m$ and $i$. 
By Lemma \ref{lemma-particular-solutions} with $m=0$ and $\gamma=\widetilde\gamma_i$ for $i=1, \cdots, l_1$,  
\eqref{eq-solving-linear1} admits a solution $\widetilde \eta_1$ 
of the form \eqref{eq-Asymptotic-Uk-22}. 
In conclusion, we obtain a function $\widetilde \eta_1$ in the form \eqref{eq-Asymptotic-Uk-22}, 
and $\widetilde v_1$ defined by \eqref{eq-Asymptotic-Uk-21} satisfies \eqref{eq-Asymptotic-Uk-23}. 
By \eqref{eq-Asymptotic-Uk-13} and \eqref{eq-Asymptotic-Uk-22}, we have 
\begin{equation}\label{eq-Asymptotic-Uk-23a}
\widetilde v_1=O(e^{-\widetilde\gamma_1t}\rho^{\nu}).
\end{equation}

{\it Case 2:  We now assume $\gamma_{k_1+1}>3\gamma_1$.} Then, $\widetilde \gamma_{l_1}\ge 3\gamma_1$. 

Let $n_1$ be the largest integer such that $\widetilde \gamma_{n_1}<3\gamma_1$. 
Then, $\widetilde \gamma_{n_1+1}=3\gamma_1$. 
We can repeat the argument in Case 1 with $n_1$ replacing $l_1$. 
In defining $I_1$ in \eqref{eq-definitioni-I}, the summation is from $i=1$ to $n_1$. 
Similarly for $\widetilde \eta_1$ in \eqref{eq-Asymptotic-Uk-22}, we define
\begin{equation}\label{eq-Asymptotic-Uk-24}\widetilde\eta_{11}(t,\theta)=
\sum_{i=1}^{n_1}e^{-\widetilde \gamma_it}w_i(\theta),\end{equation}
for appropriate  functions $w_i$, and then set
\begin{equation}\label{eq-Asymptotic-Uk-25}
\widetilde v_{11}=v_1-\widetilde\eta_{11}.
\end{equation}
A similar arguments yields
\begin{equation}\label{eq-Asymptotic-Uk-26}\mathcal L\widetilde v_{11}
=O(e^{-\widetilde \gamma_{n_1+1}t}\rho^{\beta+2\nu-2})= O(e^{-3\gamma_1t}\rho^{\beta+2\nu-2}),\end{equation}
and a similar estimate for the $L^p$-norm. 
Moreover, by \eqref{eq-Asymptotic-Uk-13} and \eqref{eq-Asymptotic-Uk-24}, 
$$\widetilde v_{11}= O(e^{-\widetilde \gamma_1t}\rho^\nu)= O(e^{-2\gamma_1t}\rho^\nu).$$
We point out there is no $\gamma_i$ between $\widetilde \gamma_1$ and 
$\widetilde\gamma_{n_1+1}$. 
Hence, by Lemma \ref{lemma-nonhomogeneous-linearized-eq-Lip}(ii), we have 
\begin{equation}\label{eq-Asymptotic-Uk-27}\widetilde v_{11}=O(e^{-3\gamma_1t}\rho^\nu).\end{equation}
Note that \eqref{eq-Asymptotic-Uk-27} improves \eqref{eq-Asymptotic-Uk-23a} and hence \eqref{eq-Asymptotic-Uk-13}. 

Now, we are in a similar situation as at the beginning of Step 2, 
with $\widetilde\gamma_{n_1+1}=3\gamma_1$ replacing $\widetilde \gamma_1=2\gamma_1$. 
If $\gamma_{k_1+1}<4\gamma_1$, we proceed as in Case 1. 
If $\gamma_{k_1+1}>4\gamma_1$, we proceed as at the beginning 
of Case 2 by taking the largest integer $n_2$ such that $\widetilde \gamma_{n_2}<4\gamma_1$. 
After finitely many steps, we reach $\widetilde \gamma_{l_1}$. 

In summary, we have $\widetilde\eta_1$ as in \eqref{eq-Asymptotic-Uk-22} and, 
by defining  
$\widetilde v_1$ by \eqref{eq-Asymptotic-Uk-21}, 
we conclude \eqref{eq-Asymptotic-Uk-23} and \eqref{eq-Asymptotic-Uk-23-L2}, 
as well as  \eqref{eq-Asymptotic-Uk-23a}. This finishes the discussion of Step 2. 

{\it Step 3.} Now we are in the same situation as in Step 1, 
with $\widetilde\gamma_{l_1+1}$ replacing $\widetilde \gamma_1$. We repeat the argument there with 
$k_1+1$, $k_2$ and $l_1+1$ replacing $1$, $k_1$ and $1$, respectively. 
Note $\gamma_{k_2}<\widetilde \gamma_{l_1+1}$. 
By \eqref{eq-Asymptotic-Uk-23}, \eqref{eq-Asymptotic-Uk-23-L2}, 
and Lemma \ref{lemma-nonhomogeneous-linearized-eq-Lip}(ii), we obtain 
$$\widetilde v_1(t,\theta)=\sum_{i=k_1+1}^{k_2}c_ie^{-\gamma_i t}\phi_i(\theta)
+O(e^{-\widetilde \gamma_{l_1+1} t}\rho^\nu),$$
where $c_i$ is a constant, for $i=k_1+1, \cdots, k_2$.  
By \eqref{eq-Asymptotic-Uk-23a}, there is no need to adjust by terms involving $e^{-\gamma_i t}$
corresponding to $i=1, \cdots, k_1$. 
Set 
\begin{equation}\label{eq-Asymptotic-Uk-31}\eta_2(t, \theta)=\sum_{i=k_1+1}^{k_2}c_ie^{-\gamma_i t}\phi_i(\theta),
\end{equation}
and 
\begin{equation}\label{eq-Asymptotic-Uk-32} 
v_2=\widetilde v_1-\eta_2.\end{equation}
Then, $\mathcal L\eta_2=0$, 
$v_2=v-\eta_1-\widetilde \eta_1-\eta_2,$  and 
\begin{equation}\label{eq-Asymptotic-Uk-33}v_2=O(e^{-\widetilde \gamma_{l_1+1} t}\rho^\nu).\end{equation}

{\it Step 4.} The discussion is similar as that in Step 2. For some $\widetilde \eta_2$ to be determined, set 
\begin{equation}\label{eq-Asymptotic-Uk-41} 
\widetilde v_2=v_2-\widetilde \eta_2.\end{equation}
Then, 
$$\mathcal L\widetilde v_2=F(v)-\mathcal L\widetilde\eta_1-\mathcal L\widetilde\eta_2.$$
Note
$$F(v)=F(v_2+\eta_1+\widetilde\eta_1+\eta_2)
=\rho^{-2-\beta}\sum_{i=2}^\infty b_i\rho^{i\beta}(v_2+\eta_1+\widetilde\eta_1+\eta_2)^i.$$
As in Step 2, we need to analyze 
$$\sum_{i=2}^\infty b_i\rho^{i\beta}(\eta_1+\widetilde\eta_1+\eta_2)^i.$$
In Step 2, by choosing $\widetilde\eta_1$ as in \eqref{eq-Asymptotic-Uk-22} appropriately, 
we used $\mathcal L\widetilde\eta_1$ to cancel the terms $e^{-\widetilde \gamma_it}$ 
in $F(v)$, for $i=1, \cdots, l_1$. 
Proceeding similarly, we can find $\widetilde \eta_2$ in the form 
\begin{equation}\label{eq-Asymptotic-Uk-42}
\widetilde\eta_2(t,\theta)=\sum_{i=l_1+1}^{l_2}
e^{-\widetilde \gamma_it}w_i(\theta)
\end{equation}
to cancel the terms $e^{-\widetilde \gamma_it}$ in $F(v)$, for $i=l_1+1, \cdots, l_2$.
By defining 
$\widetilde v_2$ by \eqref{eq-Asymptotic-Uk-41}, 
we conclude 
\begin{equation}\label{eq-Asymptotic-Uk-43}
\mathcal L\widetilde v_2=O(e^{-\widetilde \gamma_{l_2+1}t}\rho^{\beta+2\nu-2}),\end{equation}
and 
\begin{equation}\label{eq-Asymptotic-Uk-43-L2}
\|(\mathcal L\widetilde v_2)(t,\cdot)\|_{L^p(\Sigma)}=O(e^{-\widetilde \gamma_{l_2+1}t}).\end{equation}

We can continue these steps indefinitely and hence finish the proof for the case \eqref{eq-special}. 

Next, we consider the general case; namely, some $\gamma_i$ can be written 
as a linear combination of some of $\gamma_1, \cdots, \gamma_{i-1}$ with positive integer coefficients. 
We will modify discussion above to treat the general case. Whenever some $\gamma_i$ coincides 
some $\widetilde \gamma_{i'}$, an extra power of $t$ appears when solving the linear equation 
$\mathcal Lw=f$,
according to Lemma \ref{lemma-particular-solutions},  and such 
a power of $t$ will generate more powers of $t$ upon iteration. 

For an illustration, we consider $\gamma_{k_1}=\widetilde \gamma_1$ instead of the strict inequality in 
\eqref{eq-arrangement}. This is the first time that some $\gamma_i$ may coincide
some $\widetilde \gamma_{i'}$.
We start with \eqref{eq-Asymptotic-Uk-02}
and \eqref{eq-Asymptotic-Uk-03}, and proceed similarly as in Step 1. 
Take $k_*\in \{1, \cdots, k_1-1\}$ such that 
$$\gamma_{k_*}<\gamma_{k_*+1}=\cdots=\gamma_{k_1}=\widetilde \gamma_1=2\gamma_1.$$ By 
Lemma \ref{lemma-nonhomogeneous-linearized-eq-Lip}(ii), we obtain 
$$v(t, \theta)=\sum_{i=1}^{k_*}c_ie^{-\gamma_i t}\phi_i(\theta)+O(te^{-\widetilde \gamma_1 t}\rho^\nu),$$
where $c_i$ is a constant, for $i=1, \cdots, k_*$. Instead of \eqref{eq-Asymptotic-Uk-11}, we define 
\begin{equation}\label{eq-Asymptotic-Uk-11z}
\eta_1(t, \theta)=\sum_{i=1}^{k_*}c_ie^{-\gamma_i t}\phi_i(\theta),\end{equation}
and then define $v_1$ as in 
\eqref{eq-Asymptotic-Uk-12}. Then, 
\begin{equation}\label{eq-Asymptotic-Uk-13z}v_1=O(te^{-\widetilde \gamma_1 t}\rho^\nu).\end{equation}
Next, we proceed similarly as in Step 2. In the discussion of Case 1 in Step 2, we need to solve 
\eqref{eq-solving-linear1} and find $\widetilde \eta_1$, which is a linear combination of 
$e^{-\widetilde \gamma_1t}$, $\cdots$, $e^{-\widetilde \gamma_{l_1}t}$. 
For $i=2, \cdots, l_1$, the 
part corresponding to $e^{-\widetilde \gamma_it}$
is the same,  still given by $e^{-\widetilde \gamma_it} w_i(\theta)$. For $i=1$, 
the part corresponding to $e^{-\widetilde \gamma_1t}$ is given by 
\begin{equation}\label{eq-solving-linear2z}
te^{-\widetilde \gamma_1t}w_{11}(\theta)+e^{-\widetilde \gamma_1t}\phi_{k_1}(\theta),
\end{equation}
where $w_{11}\in C^\infty(\Sigma)\cap C^\nu(\bar\Sigma)\cap H_0^1(\Sigma)$ with $w_{11}=0$ on $\partial\Sigma$. 
Then, by defining  $\widetilde\eta_1$ by \eqref{eq-Asymptotic-Uk-22}, with the new  
expression given by 
\eqref{eq-solving-linear2z} for $e^{-\widetilde \gamma_1t}$, 
and defining $\widetilde v_1$ by \eqref{eq-Asymptotic-Uk-21}, we have \eqref{eq-Asymptotic-Uk-23}. 
We can modify the rest of the proof similarly. \end{proof}

According to the proof,  the summation in \eqref{eq-U2-0-z-Lip} has two sources, 
the kernel of the linearized equation and the nonlinearity. 
The kernel part is a linear combination of $e^{-\gamma_it}\phi_i$ as in 
Lemma \ref{lemma-nonhomogeneous-linearized-eq-Lip}(ii), with constant coefficients. 
The nonlinear part  consists of  
solutions constructed in Lemma \ref{lemma-particular-solutions} to eliminate nonlinear combinations 
of lower order terms in $F(v)$. 

Now we are ready to prove Theorem \ref{thrm-main1}. 

\begin{proof}[Proof of Theorem \ref{thrm-main1}]
Let $\xi\in C^\infty(\Sigma)$ be the solution of 
\eqref{eq-LN-spherical-domain-eq}-\eqref{eq-LN-spherical-domain-boundary}. 
By  \eqref{eq-U2-0-z-Lip}, 
we have 
\begin{equation*}
\Big|\xi^{-1}v-\sum_{i=1}^{m}\sum_{j=0}^{i-1}\xi^{-1}c_{ij}t^je^{-\mu_it}\Big|\le Ct^me^{-\mu_{m+1}t}\xi^{-1}\rho^{\nu}
\quad\text{in } (T+1,\infty)\times\Sigma.\end{equation*} 
By \eqref{eq-estimate-xi-upper-lower}, we get 
$$|\xi^{-1}c_{ij}|+|\xi^{-1}\rho^{\nu}|\le d^{\beta+\nu}\quad\text{in }\Sigma,$$
where $d$ is the distance function in $\Sigma$ to $\partial\Sigma$. 
We now have the desired result with $\tau=\beta+\nu$ by the definition of $v$ in \eqref{eq-def-v-spherical} 
and the change of coordinates \eqref{eq-change-coordinates}.
\end{proof}  

Under the additional assumption that  
$\rho$ also satisfies \eqref{eq-condition-gradient}, 
Theorem \ref{thrm-main-v-Lip} holds for $\tau=s$. 
Hence, we have 
Theorem \ref{thrm-main2}.

To end this paper, we make one final remark. 
Let $\Sigma$ be a smooth domain and 
$\rho$ be the positive solution of 
\eqref{eq-LN-domain-eq-rho-new}-\eqref{eq-LN-domain-rho-boundary-new}. 
Recall that $d$ is the distance function in $\Sigma$ to $\partial\Sigma$. 
Then, for any $m\ge n+1$, $\alpha\in (0,1)$, and any $\theta\in\Sigma$ near $\partial\Sigma$, 
\begin{equation}\label{eq-expansion-rho} 
\Big|\rho(\theta)-\Big[\sum_{i=1}^{n}c_i(\theta')d^i
+\sum_{i=n+1}^m \sum_{j=0}^{N_i} c_{i, j}(\theta')d^i(\log d)^j\Big]\Big|\le Cd^{m+\alpha},\end{equation}
where $d=d(\theta)$, $\theta'\in\partial\Sigma$ is the unique point with 
$d(\theta)=\mathrm{dist}(\theta,\theta')$, $N_i$ is a positive integer depending only $n$ and $i$, 
$C$ is a positive constant depending only on $n$, $m$ and $\alpha$, 
and $c_i$ and $c_{i,j}$ are smooth functions on $\partial\Sigma$.  
Refer to \cite{ACF1982CMP} and \cite{Mazzeo1991} for details. 
Similarly, let $\phi_i$ 
be the eigenfunction established in Theorem \ref{thrm-singular-eigenvalue}.
Then, for any $m\ge n+1$, $\alpha\in (0,1)$, and any $\theta\in\Sigma$  near $\partial\Sigma$, 
\begin{equation}\label{eq-expansion-eigenfunction}
\Big|\phi_l(\theta)-d^s\Big[\sum_{i=0}^{n}c_{l,i}(\theta')d^i
+\sum_{i=n+1}^m \sum_{j=0}^{N_i} c_{l,i, j}(\theta')d^i(\log d)^j\Big]\Big|\le Cd^{m+s+\alpha}.\end{equation}
Similar expansions hold for the coefficients $c_{ij}$ in \eqref{eq-U2-0-z-Lip} near $\partial\Sigma$.
As a consequence, we can expand $v$ as a series in terms of $t^je^{-\mu_it}\rho^{k+s}(\log\rho)^l$ 
with coefficients defined on $\partial\Sigma$, for positive integer $i$ and nonnegative integers 
$j$, $k$ and $l$. 
Refer to \cite{jiang} for details.


\end{document}